\documentclass[11pt]{article}
\usepackage{lmodern}
\usepackage{amsmath}
\usepackage{amsthm}
\usepackage{amssymb}
\usepackage{amsfonts}
\usepackage{enumitem}
\usepackage{fullpage}
\usepackage{hyperref}
\usepackage{tikz}
\usepackage{picins}
\usepackage{subcaption}
\usepackage{comment}

\theoremstyle{plain}
\newtheorem{theorem}{Theorem}

\newtheorem{lemma}[theorem]{Lemma}
\newtheorem{proposition}[theorem]{Proposition}
\newtheorem{conjecture}[theorem]{Conjecture}
\newtheorem{assumption}[theorem]{Assumption}
\theoremstyle{definition}
\newtheorem{definition}[theorem]{Definition}
\newtheorem{example}[theorem]{Example}
\newtheorem{remark}[theorem]{Remark}
\newtheorem{problem}[theorem]{Problem}
\newtheorem{notation}[theorem]{Notation}

\newcommand{\bC}{\mathbb{C}}

\allowdisplaybreaks
\graphicspath{{fig/}}

\begin{document}

\title{Geometry of Newton homotopies: bivariate case}
\author{Jennifer Buettner$^{1}$\thanks{%
Supported in part by NSF CCF-2331401 and CCF-2212461.} \and Jonathan D.
Hauenstein$^{2}$\thanks{%
Supported in part by NSF CCF-2331440, Simons Foundation SFM-00005696, and
the Robert and Sara Lumpkins Collegiate Professorship.} \and Caroline Hills$%
^{2}$\thanks{%
Supported in part by NSF CCF-2331440 and the Robert and Sara Lumpkins
Collegiate Professorship.} \and Hoon Hong$^{1}$\thanks{%
Supported in part by NSF CCF-2331401 and CCF-2212461.} \and Francisco Ponce
Carrion$^{1}$\thanks{%
Supported in part by NSF CCF-2212461.} \and Emma L. Schmidt$^{2}$\thanks{%
Supported in part by NSF CCF-2331440 and the Arthur J. Schmitt Presidential
Leadership Fellowship.}}
\date{$^{1}$Department of Mathematics, North Carolina State University,
Raleigh, NC 27695, \texttt{\{jdbuettn,hong,fmponcec\}@ncsu.edu} \\
\smallskip$^{2}$Department of Applied and Computational Mathematics and
Statistics, University of Notre Dame, Notre Dame, IN 46556, \texttt{%
\{hauenstein,eschmi23\}@nd.edu, chills1@alumni.nd.edu}}
\maketitle

\begin{abstract}
\noindent A standard question in computational real algebraic geometry is to compute all
real solutions to a system of polynomial equations with real coefficients.
One classical and promising approach is to track along a connected component
of a real curve defined by a Newton homotopy, which is dependent upon the
selected start point. As the start point varies, different subsets of real
solutions may be obtained. This yields a partition of the space of start
points into cells, and it is important to understand the structure of this
partition in order to develop efficient algorithms based on Newton
homotopies. The structure of the boundary of such cells and the number of
cells in the corresponding partition are investigated for bivariate systems.
Several examples are included to demonstrate the results. \newline

\noindent Keywords: Newton homotopies, real solutions, bivariate systems,
homotopy continuation, real numerical algebraic geometry
\end{abstract}

\section{Introduction}

\label{sec:intro}

Given polynomials $f_1,\dots,f_n\in\mathbb{R}[x_1,\dots,x_n]$, a standard
question in computational real algebraic geometry is to find the solution set in $%
\mathbb{R}^n$ to $f_1 = \cdots = f_n = 0$, namely 
\begin{equation*}
V_\mathbb{R}(f_1,\dots,f_n) = \{x\in\mathbb{R}^n~|~f_1(x) = \cdots = f_n(x)
= 0\}.
\end{equation*}
As usual, we will assume that the corresponding solution set in $%
\mathbb{C}^n$, namely 
\begin{equation*}
V_\mathbb{C}(f_1,\dots,f_n) = \{x\in\mathbb{C}^n~|~f_1(x) = \cdots = f_n(x)
= 0\},
\end{equation*}
is finite, i.e., $V_\mathbb{C}(f_1,\dots,f_n)$ is zero dimensional. 

One obvious approach to find all the real solutions could be to find all of the complex
solutions and then sort through to select the real solutions. The
downside is that the number of real solutions may be small compared with the
number of complex solutions. For example, 
counting multiplicity in the univariate case, the fundamental theorem of algebra posits that the number of
complex solutions is the degree while Descartes' rule of signs can be used
to bound the number of real solutions based on the number of terms.

A classical (e.g., already commonly used in 1968, see \cite{Meyer1968}) and
promising approach for computing only the real solutions is via Newton
homotopies. Given a start point $r\in\mathbb{R}^{n}$ with $f(r)\neq0$, the corresponding
Newton homotopy is 
\begin{equation*}
H(x,r;t) = f(x) - t\cdot f(r) = 0. 
\end{equation*}
Some applications of Newton homtopies include global optimization~\cite%
{TrajectoryGlobalOptimization}, potential energy landscapes~\cite%
{SamplingNewtonHomotopy,ThomsonProblemNewtonHom,MechanochemicalNewton,BifurcationNewtonHomotopy}%
, and decomposing positive-dimensional solution sets using monodromy~\cite%
{UsingMonodromy}. The simple structure of Newton homotopies is conducive to
efficient certified path trackers~\cite%
{CertificationNewtonHomotopy,CertifiedTrackingNewtonHomotopy}. Moreover, the
geometry of Newton homotopies for bivariate quadratics was studied in \cite%
{Newton22case}.

When $n=1$, the Newton homotopy is very special. For a
given~$r\in\mathbb{R}$ with~$f(r)\neq0$, the real solution set of the Newton
homotopy $H$ corresponds with the graph of~$f$ over~$\mathbb{R}$, which is
connected. Since the solutions of $f=0$ are precisely where $t=0$, one can
start from~$r$ at~$t=1$ and track along the graph of $f$ as a curve in $%
(x,t)\in\mathbb{R}^{2}$ to obtain every real solution~of~$f=0$.

When $n\geq2$, the real solution set of the Newton homotopy need not be
connected. Let $Z_{r}$ denote the set of real roots of $f$ on the same
connected component as $r$. This naturally induces a partition on the space
of start points into cells, a path connected region where $Z_{r} =
Z_{r^{\prime}}$ when~$r$ and $r^{\prime}$ are in the same cell. Therefore,
in the study of Newton homotopies, it is imperative to understand the
structure of such a partition for $n\geq 2$.

One naturally should begin with $n = 2$ in the hope of gaining insights to tackle larger $n$ and
eventually arbitrary $n$. 
In this paper, we focus on $n = 2$ since this is already nontrivial
and interesting.  
Our main contributions are an
investigation of the boundaries of the cells 
(Theorems~\ref{thm:geocharacterization},~\ref{thm:DefiningPolynomialsBdry},~and~\ref{thm:tangentcone}) and a bound on the number of cells in the
partition (Theorems~\ref{thm:num_comp_equal}~and~\ref{thm:num_comp}). Although we tackle fundamental structural questions related to Newton homotopies here,
we hope this foundational study will guide and inspire development of
efficient algorithms for computing the real~solutions.

The rest of the paper is structured as follows. Section~\ref{sec:problem}
precisely states the problems under consideration in this paper while Section~%
\ref{sec:Assumptions} summarizes the genericity assumptions used throughout.
Section~\ref{sec:geocharacterization} looks at the geometry of the boundary.
Section~\ref{sec:singPaths} algebraically describes the extended boundary of
the cells which arise from singularities. The tangent cone at these
singularities are investigated in Sections~\ref{sec:TangentCone}. The number
of singularities and the number of cells are considered in Section~\ref%
{sec:NumberConnected}, along with several examples. A short conclusion is
provided in Section~\ref{sec:Conclusion}.

\section{Problems}

\label{sec:problem} In this section, we will identify and precisely state
some fundamental problems that we will tackle in this paper. We first recall
or introduce various notions and illustrate them with the following
running~example that will use throughout.

\newpage 

\begin{example}[Running Example]
\label{runningExample}\ \pichskip{0pt} 
\parpic[r][t]{
\begin{minipage}{0.27\linewidth}
\includegraphics[width=\linewidth]{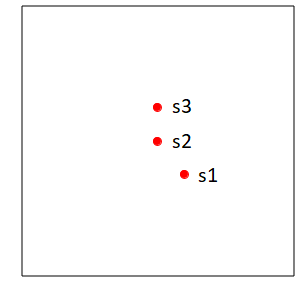}
\end{minipage}
} \noindent The following system $f$ will be used as a running
illustrative example: 
\begin{equation*}
f = 
\begin{bmatrix}
x_{1}^{3}-x_{2}^{2}-x_{1} x_{2}+1 \\ 
x_{2}^{3}-x_{1}^{2}+x_{1} x_{2}+1%
\end{bmatrix}%
. 
\end{equation*}
It is known that $V_\mathbb{C}(f)$ consists of $9$ points while $V_\mathbb{R}(f)
= \{s_1,s_2,s_3\}$ where
\begin{center}
\begin{tabular}{ccc}
$s_{1}=(0,-1),$ & $s_{2}=(-1,0),$ & $s_{3}=(-1,1)$%
\end{tabular}
\end{center}
as shown in the figure on the right.
\end{example}

\bigskip

\begin{definition}[Newton homotopy]
\label{def:NewtonHomotopy}The \emph{Newton homotopy} $H$ of $f\in\mathbb{R}%
\left[ x_{1},\ldots,x_{n}\right] ^{n}$ with \mbox{$r\in \mathbb{R}^{n}$} is defined
by%
\begin{equation*}
H\left( x,r;t\right) =f\left( x\right) -t\ f\left( r\right) . 
\end{equation*}
\end{definition}

\bigskip

\begin{example}[Running Example continued]
\label{ex:f} For the running example (Example~\ref{runningExample}), the
Newton homotopy $H$ of $f$ is 
\begin{equation*}
H(x,r;t) = 
\begin{bmatrix}
x_{1}^{3}-x_{2}^{2}-x_{1} x_{2}+1 - t( r_{1}^{3}-r_{2}^{2}-r_{1} r_{2}+1) \\ 
x_{2}^{3}-x_{1}^{2}+x_{1} x_{2}+1 - t(r_{2}^{3}-r_{1}^{2}+r_{1} r_{2}+1)%
\end{bmatrix}
. 
\end{equation*}
\end{example}

\bigskip

\begin{definition}[Newton homotopy curve]
\label{def:NewtonHomotopyCurve} For $f\in\mathbb{R}\left[ x_{1},\ldots,x_{n}%
\right] ^{n}$ and $r\in \mathbb{R}^n\setminus V_\mathbb{R}(f)$, the \emph{%
Newton homotopy curve} $C_{r}$ of $f$ and $r$ is defined~by%
\begin{equation*}
C_{r}=\left\{ x\in\mathbb{R}^{n}: \ \underset{t\in\mathbb{R}}{\exists}(x,t)
\in V_{r} \right\} \hbox{~~where~~} V_{r} = \left\{ (x,t) \in\mathbb{R}%
^{n+1}: \ H\left( x,r;t\right) =0\right\}. 
\end{equation*}
\end{definition}

\bigskip

\begin{example}[Running Example continued]
For the running example (Example~\ref{runningExample}),
consider the Newton homotopy with
\mbox{$r = (1.47,2)$}.
The figure on the left depicts~$V_{r}\subset\mathbb{R}^3$. By projecting $%
V_{r}$ onto $x$-space, we obtain $C_{r}\subset\mathbb{R}^2$ as shown in the
figure below.

\begin{center}
\begin{tabular}{ccc}
\includegraphics[scale=0.60]{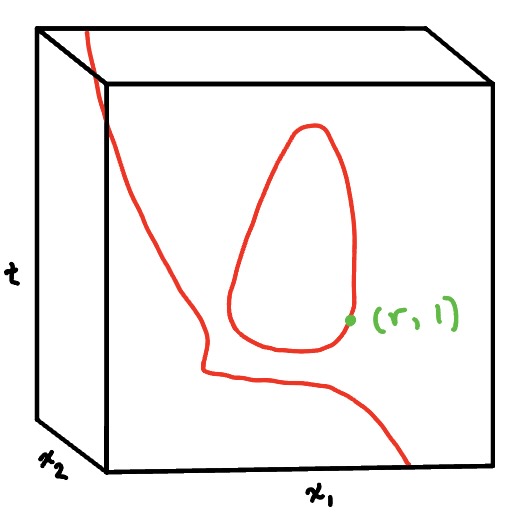} & 
\hspace{1cm} & %
\includegraphics[scale=0.60]{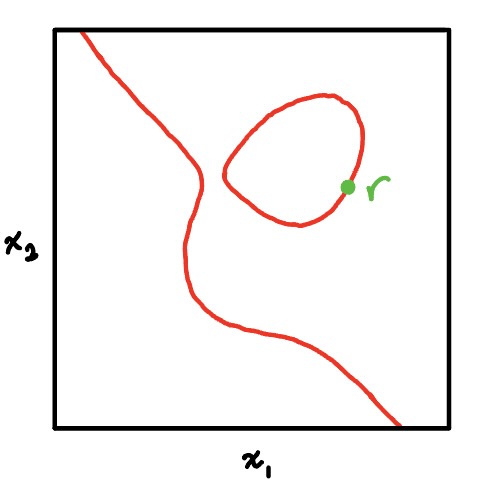} \\ 
$V_{r}$ where $r=(1.47,2)$ & \hspace{1cm} & $C_{r}$ where $r=(1.47,2)$%
\end{tabular}
\end{center}
\end{example}

\newpage
\begin{definition}[Newton homotopy start component]
\label{def:NewtonHomotopyStartCurve}The \emph{Newton homotopy start component%
} $\zeta_{r}$ of~$f$ with $r$ is the maximally connected component of $C_{r}$
that contains $r$.
\end{definition}
\begin{example}[Running Example continued]
For the running example (Example~\ref{runningExample}), the figure on the
left depicts $\zeta_r$ for $r=(1.47,2)$ while the figure on the right
depicts $\zeta_r$ for $r=(-0.843,2.25)$.

\begin{center}
\begin{tabular}{ccc}
\includegraphics[width=0.16%
\linewidth]{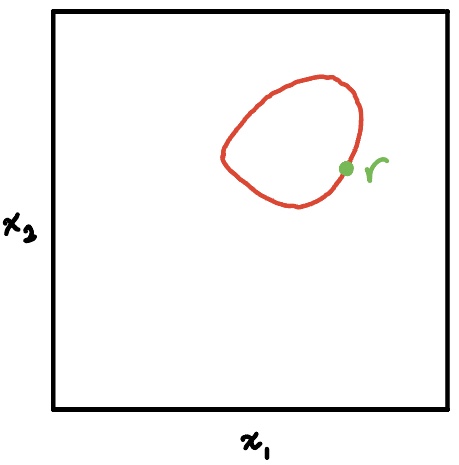} & \hspace{1cm} & %
\includegraphics[width=0.16%
\linewidth]{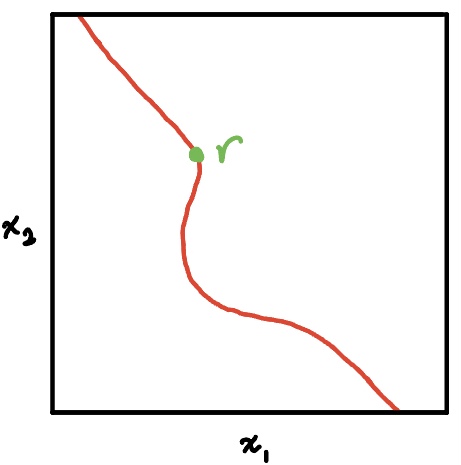} \\ 
$r = (1.47,2)$ & \hspace{1cm} & $r = (-0.843,2.25)$%
\end{tabular}
\end{center}
\end{example}

Clearly, $V_\mathbb{R}(f)\subset C_r$. As summarized in the Introduction, in
the univariate case, $C_r = \zeta_r$ for every $r\in\mathbb{R}\setminus V_%
\mathbb{R}(f)$ so that one obtains every point in $V_\mathbb{R}(f)$ by
simply tracking along $C_r=\zeta_r$. Since this does not necessarily hold
when $n\geq2$, we note that the set $V_\mathbb{R}(f)\cap \zeta_r$ is the set
of points in $V_\mathbb{R}(f)$ that are obtained by tracking along $\zeta_r$
starting from $r$.

\medskip

For a subset of $V_\mathbb{R}(f)$, the following describes all of the points 
$r$ for which the start component~$\zeta_r$ contains exactly that subset of
real solutions.

\begin{definition}[Class]
\label{def:class} For $\upsilon\subset V_{\mathbb{R}}(f),$ the \emph{class}
of $\upsilon$, denoted as ${Cl}_{\upsilon}$, is defined as 
\begin{equation*}
{Cl}_{\upsilon}\;\; =\;\; \{r \in\mathbb{R}^{n} \setminus V_{\mathbb{R}}(f)
\;:\; V_{\mathbb{R}}(f) \cap \zeta_{r} = \upsilon\}. 
\end{equation*}
\end{definition}

\begin{example}[Running Example continued]
\ The following illustrates the various classes for the running example
(Example~\ref{runningExample}) along with a start component from each of the
four nonempty classes which are colored as follows:

\noindent 
\begin{center}
\begin{tabular}{|c|c|c|c|c|c|c|c|c|}
\hline
$v$ & \{\} & \{$s_{1}$\} & \{$s_{2}$\} & \{$s_{3}$\} & \{$s_{1},s_{2}$\} & \{%
$s_{1},s_{3}$\} & \{$s_{2},s_{3}$\} & \{$s_{1},s_{2},s_{3}$\} \\ \hline
$Cl_v$ & white & yellow & \{\} & \{\} & \{\} & \{\} & pink & cyan \\ \hline
\end{tabular}

\noindent 
\begin{tabular}{cc|ccccc}
\includegraphics[width=0.16\linewidth]{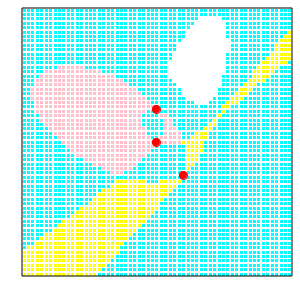} &  &  & %
\includegraphics[width=0.16\linewidth]{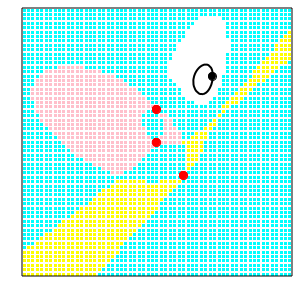} & %
\includegraphics[width=0.16\linewidth]{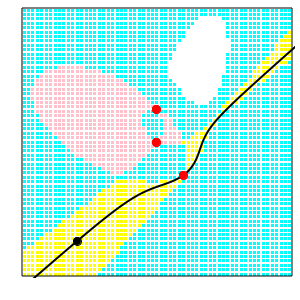} & %
\includegraphics[width=0.16\linewidth]{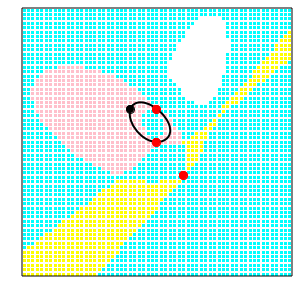} & %
\includegraphics[width=0.16\linewidth]{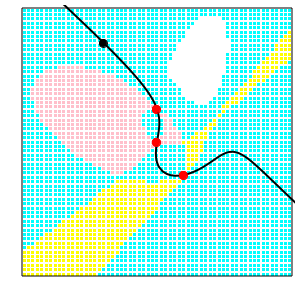}%
\end{tabular}
\end{center}
\end{example}

One key object of study is the boundary of the classes, which is defined
next.

\begin{definition}[Boundary]
\label{def:boundary} The \emph{boundary} of the classes, denoted as $B$, is
defined by 
\pichskip{0pt} 
\parpic[r][t]{
\includegraphics[width=0.17\linewidth]{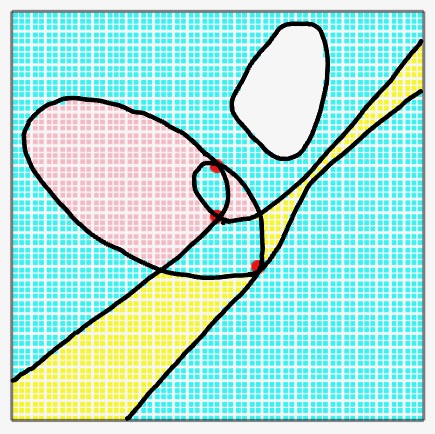}
}

\begin{equation*}
B = \underset{\upsilon\subset V_{\mathbb{R}}(f)}{\bigcup} \partial {Cl}%
_{\upsilon} \setminus V_{\mathbb{R}}(f) 
\end{equation*}
\end{definition}
\begin{example}[Running Example continued]
The boundary $B$ is shown in \\ black for the the running example (Example~\ref%
{runningExample}) in the figure on the right.
\end{example}

\bigskip

In Definition~\ref{def:boundary}, the boundary $B$ is characterized using
point-set topology. Since the underlying Newton homotopy curves are defined
in terms of polynomial equations, we aim to characterize the boundary~$B$
using algebraic geometry. For this, as usual, we will characterize $%
\overline{B}$, the Zariski closure~of~$B$.

\begin{definition}[Extended Boundary]
\label{def:extended boundary} The \emph{extended boundary}, denoted as $%
\overline{B}$, is defined as 
\begin{equation*}
\overline{B} = (\text{Zariski Closure of }B) \cap\mathbb{R}^{n} 
\end{equation*}
Equivalently, it can be defined as 
\begin{equation*}
\overline{B} = \text{Real Zariski Closure of } B 
\end{equation*}
where Real Zariski Closure of $B$ stands for the intersection of all real
algebraic supersets of $B$.
\end{definition}

\begin{example}[Running Example continued]\label{ex:BoundaryPicture}
For the running example (Example~\ref{runningExample}), the extended
boundary is given by the blue curve in the figure below.

\begin{center}
\includegraphics[width=0.25\linewidth]{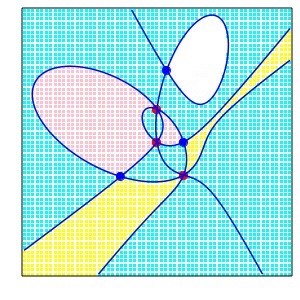}
\end{center}
\end{example}

The extended boundary $\overline{B}$ is an algebraic object whose complement
consists of connected components, which are objects of interest.

\begin{definition}[Cells]
The \emph{cells} are the maximally connected components of $\mathbb{R}^{n}
\setminus\overline{B}.$
\end{definition}

\begin{example}[Running Example continued]
For the running example (Example~\ref{runningExample}), the figure below
labels the $13$ cells obtained by removing the blue curve, the extended
boundary, from the real plane.

\begin{center}
\includegraphics[width=0.25\linewidth]{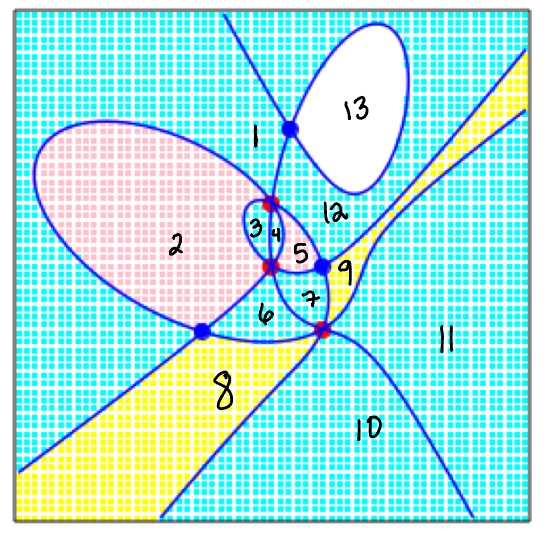}
\end{center}

\noindent Since $B\subset \overline{B}$, each connected component of $\mathbb{R}%
^n\setminus\overline{B}$ is contained in a connected component of $\mathbb{R}%
^n\setminus B$. Moreover, each maximally connected component of $\mathbb{R}%
^n\setminus B$ arises as the union of cells along with parts of $\overline{B}%
\setminus B$. For example, cells $10$ and~$11$ together with their common
boundary would form a maximally connected component of $\mathbb{R}%
^{2}\setminus B$. Similar statements hold for the pairs $1$ and~$12$, $3$
and~$4$, and $6$ and~$7$. In particular, $\mathbb{R}^{2}\setminus B$ has $%
13-4=9$ maximally connected components.
\end{example}

\newpage

Using the above characterizations, we now state a few problems that we will
consider throughout the rest of the paper.

\begin{problem}
\label{problem:geocharacterization} \emph{Find a geometric characterization
of the boundary.}
\end{problem}

\begin{problem}
\label{problem:algcharacterization} \emph{Find an algebraic characterization
of the extended boundary -- that is, find polynomials that vanish on the
extended boundary.}
\end{problem}

\begin{problem}
\label{problem:localStructure} \emph{Characterize the local structure of the
extended boundary at singularities.}
\end{problem}

\begin{problem}
\label{prob:NumCells} \emph{Find a ``reasonable'' upper bound on the number
of cells.}
\end{problem}

\medskip

\noindent {\bf Focus on bivariate systems:} The above problems are either trivial or meaningless when $n=1$.
Thus, one naturally should begin with $n=2$ in the hope of gaining insights
to tackle larger $n$ and eventually arbitrary $n$. This paper will focus on $n=2$ since it is already non-trivial and interesting. 
Bivariate systems are well-studied and arise naturally in various contexts
such as computer graphics~\cite{ManochaDemmel} and Gale dual transformations
of certain fewnomial systems~\cite{FewnomialSoftware,GaleDual}. Some other
approaches for computing real solutions to bivariate systems are
Sturm-Habicht sequences~\cite{RealBivariate} and generalized eigenvalue
decomposition~\cite{BivariateSolving}. Nonetheless, the focus here is
foundational in terms of analyzing Newton homotopies 
for solving bivariate systems summarized in the problems~listed~above.

\section{Generic assumptions}
\label{sec:Assumptions}

Let $f_{1},f_{2}\in\mathbb{R}[x_{1},x_{2}]$ such that $d_{j} = \deg f_{j} > 0
$ where $d_{1} \geq d_{2}$ and $d_{1}\geq2$. Although we consider~$f_{1}$
and~$f_{2}$ with real coefficients, we assume that $f_{1}$ and~$f_{2}$
satisfy the following genericity conditions over the complex~numbers
throughout the rest of the paper.

\begin{assumption}[Genericity conditions]
\label{assumption} \ 

\begin{enumerate}
\item \label{smooth} $V_\mathbb{C}(f_{j})$ is a smooth curve of degree $d_{j}
$.

\item \label{smooth-homo} $f_{j}^{\ast}=0$ has $d_{j}$ nonsingular solutions
in the complex projective space, where $f_{j}^{\ast}$ denotes the
homogeneous part of $f_{j}$ of highest degree (which has degree $d_{j}$).

\item \label{com-int} $V_\mathbb{C}(f_{1},f_{2})$ is a complete intersection
with $d_{1}d_{2}$ nonsingular solutions in $\mathbb{C}^{2}$.

\item \label{lin-com} 
For $r\in\bC^2\setminus V_\bC(f)$ and 
$C_r^\bC = \left\{x\in\bC^n : \ \underset{t\in\mathbb{\bC}}{\exists}(x,t) \ H\left( x,r;t\right) =0\right\}$, 
the complex Newton homotopy curve:

\begin{enumerate}
\item \label{lin:deg} $C_r^\bC$ is a curve of degree $d_1$ whenever $f_2(r)\neq 0$;

\item \label{lin:inf} $C_r^\bC$ is smooth at infinity except possibly when 
$f_2(r) = 0$ and $d_{1}>d_{2}$;

\item \label{lin:sing} $C_r^\bC$ has at most one singularity 
and that singularity (if it exists) is an ordinary double point, in other
words, a quadratic singularity.
\end{enumerate}
\end{enumerate}
\end{assumption}

\noindent The above assumptions are \emph{generic} in that almost all $f\in%
\mathbb{R}[x_1,x_2]^2$ satisfy all them. Items \ref{smooth}, \ref%
{smooth-homo}, \ref{com-int}, and~\ref{lin:deg} are well known generic
properties. Item~\ref{lin:inf} follows from Items~\ref{smooth-homo} and~\ref%
{lin:deg}. The genericity of Item~\ref{lin:sing} follows from \cite[Thm.~22 \&
Cor.~27]{BranchPoints} arising from the well known Lefschetz~pencil.


\section{Geometric characterization of boundary}
\label{sec:geocharacterization}

\noindent The following tackles Problem~\ref{problem:geocharacterization} by determining a geometric characterization of the boundary.

\begin{theorem}[Geometric characterization of boundary]
\label{thm:geocharacterization} We have

\begin{description}
\item[C1:] If $d_{1}=d_{2}$, then\ \ \ $B\ \ \subset\ \ \bigcup
\limits_{r:\,C_{r}\ \text{is singular}}C_{r}.$

\item[C2:] If $d_{1}>d_{2}$, then\ \ \ $B\ \ \subset\ \ \bigcup
\limits_{r:\,C_{r}\ \text{is singular}}C_{r}\ \ \ \ \cup\ \ \ \ V_{\mathbb{R}%
}(f_{2})$.
\end{description}
\end{theorem}

\medskip

\noindent Before we show the proof of  the above theorem, we will first illustrate what the  theorem is saying using two examples. 

\medskip

\begin{example}
\label{ex:generic_bound} \mbox{}\ 

\begin{enumerate}
\item For the running example (Example~\ref{runningExample}), $d_1 = d_2 = 3$%
. There are four values of $r\in\mathbb{R}^2$ such that $C_r$ is singular:
the three blue points and the green point in the figure below. 
\begin{equation*}
\includegraphics[width=0.3\linewidth]{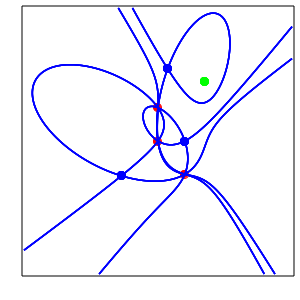} \label{fig:ex22_pt1}
\end{equation*}
The Newton homotopy curves arising from the three blue points, which are
singular points that are non-isolated over the real numbers, contain the
boundary $B$, and actually form $\overline{B}$
as in Example~\ref{ex:BoundaryPicture}. The green point is a singular
point that is isolated over the real numbers and does not contribute to $B$;
rather, it arises as the limit of a loop that contracts to a point.

\item Consider the system 
\begin{equation*}
f = \left[ 
\begin{array}{c}
x_{1}^{3} + 3x_{1}x_{2}^{2} + x_{2}^{3} + 3x_{2} - 1 \\ 
x_{1}^{2} - x_{2}^{2} + 1%
\end{array}
\right] 
\end{equation*}
where $d_{1}=3$ and $d_{2}=2$ with $V_\mathbb{R}(f)$ consisting of two red
points in the figure below. 
\begin{equation*}
\includegraphics[width=0.3\linewidth]{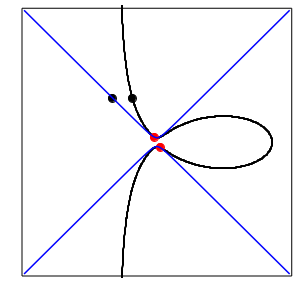} 
\end{equation*}
The Newton homotopy curves associated with the two black points are shown in
the above image. For the point $r\in V_\mathbb{R}(f_2)$, the curve $C_r$ is
a blue hyperbola (degree $2$) with 2 connected components, each containing
one point in $V_\mathbb{R}(f)$. However, for a slight perturbation, the
Newton homotopy curve is the black connected cubic curve that passes through
both points in $V_\mathbb{R}(f)$ showing that the boundary $B$ can be
contained in $V_\mathbb{R}(f_2)$ when $d_1 > d_2$.
\end{enumerate}
\end{example}

\medskip

\noindent Now we will begin to prove Theorem~\ref{thm:geocharacterization}. For this,
we need a few notations, definitions, and lemmas.

\begin{notation} \  
\begin{enumerate}
\item $Z_{r}=V_{\mathbb{R}}(f)\cap\zeta_{r}$ \hspace{4em}\textquotedblleft
the real solutions of $f$ on the start component of $C_{r}$\textquotedblright

\item $\overline{Z}_{r}=V_{\mathbb{R}}\left( f\right) \backslash\;\zeta_{r} 
\hspace{4.3em}$\textquotedblleft the real solutions of $f$ on the other
components of $C_{r}$\textquotedblright

\item $\overline{\zeta}_{r}=C_{r}\backslash\zeta_{r}$ \hspace{6.3em}%
\textquotedblleft the union of the other components\textquotedblright
\end{enumerate}
\end{notation}

\begin{example}[Running Example continued]
The previous notation is illustrated in the figure below using the running
example (Example~\ref{runningExample}). For the choice of $r$, $C_r$ has two
connected components with $\zeta_{r}$ being the one containing~$r$ and the
other being $\overline{\zeta_r} = C_r\setminus\zeta_r$. Although $V_\mathbb{R%
}(f)\subset C_r$, $Z_r = V_\mathbb{R}(f)\cap\zeta_r$ consists of two points
while $\overline{Z}_r$ consists of the other one. 
\begin{equation*}
\includegraphics[width=0.3\linewidth]{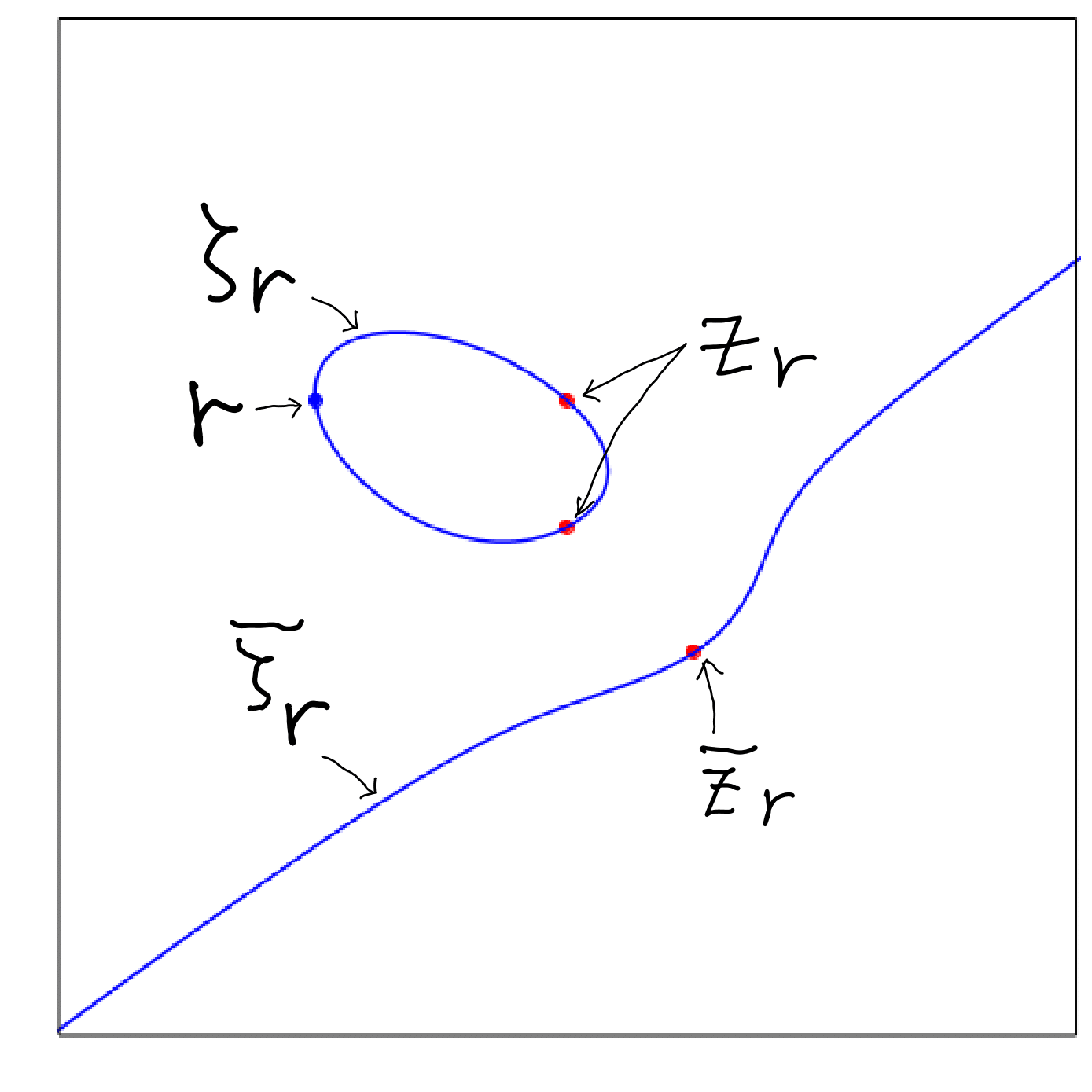} 
\end{equation*}
\end{example}

\begin{definition}[Distances]
\label{def:distances} \ \ 

\begin{itemize}
\item For $p\in\mathbb{R}^{2}$, let $\hat{p}=\frac{p }{\left\vert \left\vert
(p,1)\right\vert \right\vert _{2}}\in\mathbb{R}^{2}$. This is essentially
the Poincar\'{e} compactification.

\item For $a,b\in\mathbb{R}^{2},\ $let $d\left( a,b\right) =\left\vert
\left\vert \hat{a}-\hat{b}\right\vert \right\vert _{2}$. Lemma~\ref%
{lemma:distancemetric} shows that this is a distance function.

\item For $A,B\subset\mathbb{R}^{2}$, let

\begin{enumerate}
\item $d_{\inf}\left( A,B\right) =\inf\limits_{\left( a,b\right) \in A\times
B}d\left( a,b\right)$,

\item $d_{\sup}\left( A,B\right) =\sup\left\{ \sup\limits_{a\in
A}\inf\limits_{b\in B}d\left( a,b\right) ,\ \sup\limits_{b\in
B}\inf\limits_{a\in A}d\left( a,b\right) \right\} $ which is the Hausdorff
distance between $A$ and $B$.
\end{enumerate}

Since we are working with nonnegative numbers, we define $\sup \emptyset = 0$
and $\inf \emptyset = \infty$.
\end{itemize}
\end{definition}

\begin{lemma}
\label{lemma:distancemetric} The function $d\left( a,b\right) =\left\vert
\left\vert \hat{a}-\hat{b}\right\vert \right\vert _{2} $ is a distance
function satisfying the metric axioms.
\end{lemma}

\begin{proof}
Let $a,b\in\mathbb{R}^2$. First, we trivially have $d(a,a) = \left\vert
\left\vert \hat{a}-\hat{a}\right\vert \right\vert _{2} = \left\vert
\left\vert 0\right\vert \right\vert _{2} = 0$. Moreover, if $a\neq b$, it is
easy to see $\hat{a}\neq \hat{b}$ so that $d(a,b) >0$. Symmetry and the
triangle inequality follow from the same properties of $\|\cdot\|_2$.
\end{proof}

Since Hausdorff distance $d_{\sup}$ is a metric, $d_{\sup}$ satisfies the
triangle inequality. Although $d_{\inf}$ need not satisfy the triangle
inequality, the following shows that there is mixed triangle inequality
involving $d_{\inf}$ and $d_{\sup}$.

\begin{lemma}[Mixed Triangle Inequality]
\label{lem:inequality} If $U,V,W\subset \mathbb{R}^{2}$ are nonempty, then 
\begin{equation*}
d_{\inf}(U,W) \ \ \leq \ \ d_{\sup}(U,V)\;+\;d_{\inf}(V,W). 
\end{equation*}
\end{lemma}

\begin{proof}
The diagram on the right provides a visual reference. 
In order to yield the result,
we show the following
equivalent statement holds, which arises using a standard technique
involving perturbations to prove inequalities involving
infima and suprema:
\begin{equation*}
\underset{\delta >0}{\forall }\left( d_{\inf }(U,W)-\delta \ \ \leq \ \
d_{\sup }(U,V)\;+\;d_{\inf }(V,W)\right) .
\end{equation*}%
\vspace{-0.5cm} 
\begin{parpic}[r]{\includegraphics[width=0.45\linewidth]{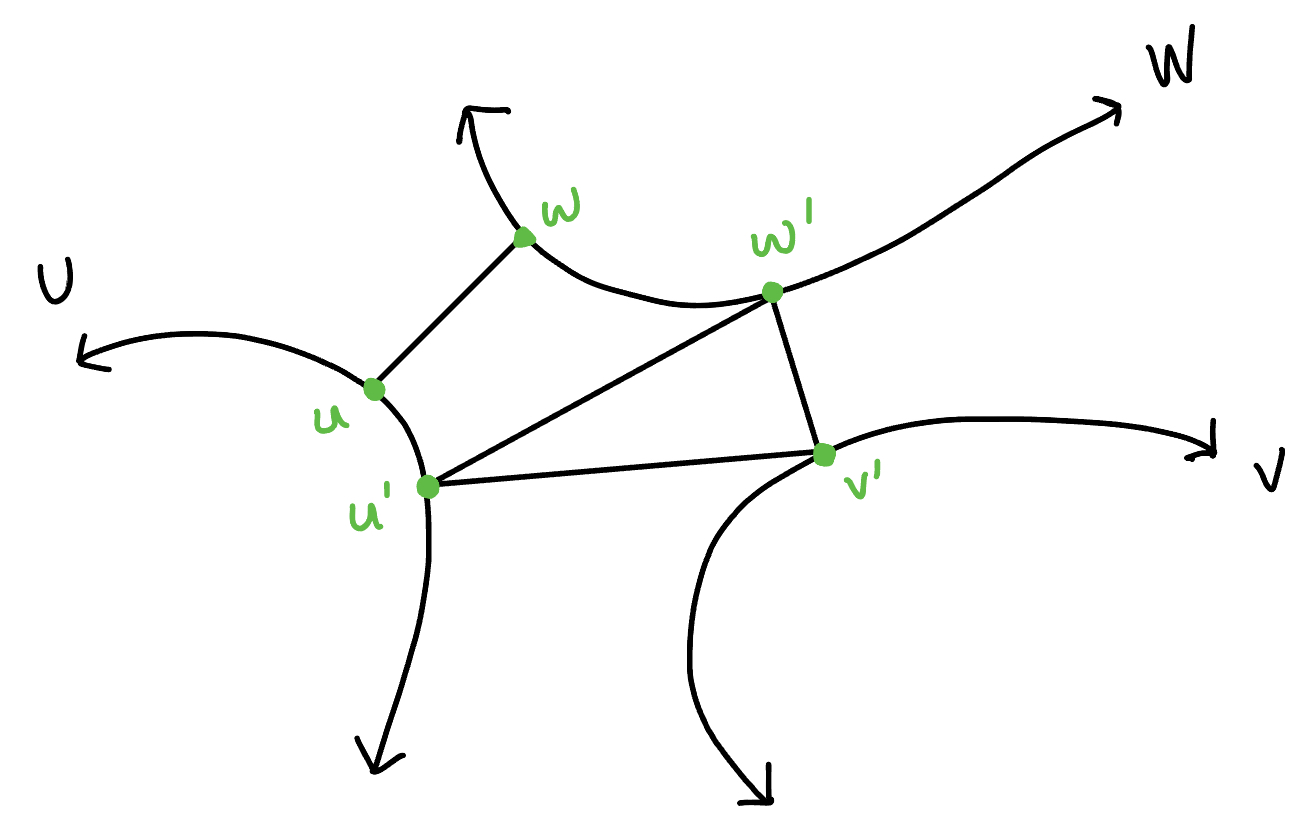}}
\noindent To that end, let $\delta > 0$.  Since $U,V,W$ are non\-empty, there exists
$u,u^{\prime} \in U$, $v^{\prime} \in V$, and $w,w^{\prime} \in W$ such~that
\begin{alignat}{3}
 d_{\inf}(U,W) & \leq d(u,w) && \leq  d_{\inf}(U,W) + \delta/4,     \label{not:1}\\
 d_{\inf}(V,W) & \leq d(v^{\prime},w^{\prime}) && \leq d_{\inf}(V,W)+\delta/4, \label{not:2}\\
 d_{\inf}(U,v^{\prime}) & \leq d(u^{\prime},v^{\prime}) && \leq d_{\inf}(U,v^{\prime}) +\delta/4. \label{not:3}
 \end{alignat}
 Hence, we have
\end{parpic}%
\begin{align*}
d_{\inf }\left( U,W\right) -\delta & \;\;\leq \;\;d\left( u,w\right)
-3\delta /4 & & \text{from }(\ref{not:1}) \\
& \;\;\leq \;\;d\left( u^{\prime },w^{\prime }\right) -\delta /2 & & \text{%
since }d(u,w)-\delta /4\leq d_{\inf }\left( U,W\right) \leq d\left(
u^{\prime },w^{\prime }\right)  \\
& \;\;\leq \;\;d\left( u^{\prime },v^{\prime }\right) +d\left( v^{\prime
},w^{\prime }\right) -\delta /2 & & \text{from the triangle inequality of }d
\\
& \;\;\leq \;\;d_{\inf }(U,v^{\prime })+d_{\inf }\left( V,W\right)  & & 
\text{from }(\ref{not:2})\text{ and }(\ref{not:3}) \\
& \;\;\leq \;\;d_{\sup }\left( U,V\right) +d_{\inf }\left( V,W\right)  & & 
\text{since }d_{\inf }(U,v^{\prime })\leq d_{\sup }\left( U,V\right) .
\end{align*}
\end{proof}

\noindent Now we state and prove the \emph{key} lemma, 
from which  Theorem~\ref{thm:geocharacterization} will follow immediately.

\begin{lemma}[Key Lemma]
\label{lem:key}
Let $p \in \mathbb{R}^2 \setminus V_{\mathbb{R}}(f)$ be such that $C_p$ is smooth at infinity. If $p\in B$, then~$C_{p}~\text{is singular.}$
\end{lemma}
\begin{proof}
Let $p \in \mathbb{R}^2 \setminus V_{\mathbb{R}}(f)$ be such that $C_p$ is smooth at infinity. 
It suffices to show that if $p\in B$, then  $C_{p}~\text{is singular.}$
We will prove the contrapositive: 
Assume  $C_{p}\ \text{is nonsingular}$ and show  $p\notin B$. 
If all points $q$ sufficiently close to $p$ lie in the same class, that is, $%
Z_p = Z_q$, then $p \notin B$. Hence, it suffices
to show 
\begin{equation}  \label{eq:Zp=Zq}
\underset{\varepsilon>0}{\exists}\ \ \underset{q\neq p}{\forall}\ \
\left\vert \left\vert p-q\right\vert \right\vert _{2}<\varepsilon\
\Longrightarrow \ Z_{p}=Z_{q}.
\end{equation}

\begin{parpic}[r]{\includegraphics[width=0.28\linewidth]{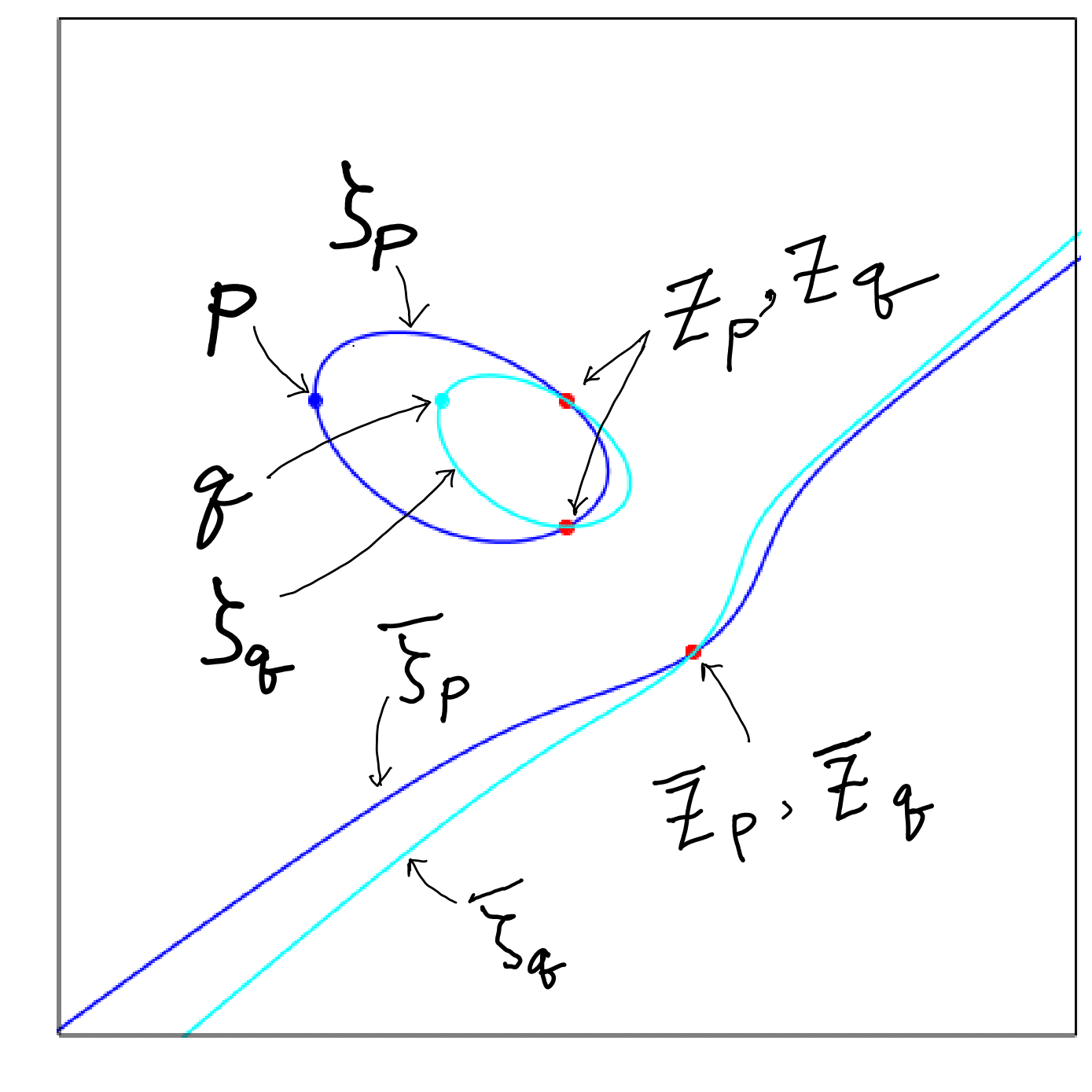}}
\noindent The annotated plot on the right can help  aid in the geometric
meaning of the following argument.

\medskip
\noindent Recall that $C_p$ is nonsingular and smooth at infinity. 
Thus, there is a one-to-one correspondence between the connected
components of $C_p$ and $C_q$ for all $q$ sufficiently close to $p$.

\medskip
\noindent In fact, by using the scaled
distance metric~$d$ in
Definition~\ref{def:distances},
which essentially provides a Poincar\'{e} compactification, 
continuity ensures that the corresponding 
connected components of
$C_p$ and $C_q$ do
not vary too much when $q$ is close to~$p$. That is,
\begin{equation}  \label{eq:DistanceComponents}
\underset{\delta>0}{\forall}\ \ \underset{\varepsilon >0}{\exists}\ \ 
\underset{q\neq p}{\forall}\ \ \left( \ \left\vert \left\vert p-q\right\vert
\right\vert _{2}<\varepsilon\ \ \ \Longrightarrow\ \ \ d_{\sup }\left(
\zeta_{q},\zeta_{p}\right) <\delta \  \land \  d_{\sup }\left( \overline{%
\zeta}_{q},\overline{\zeta}_{p}\right) <\delta\right).
\end{equation}
\end{parpic}

\noindent Now, the finiteness of $Z_p$ and $\overline{Z}_p$ along with disjointness
ensures that 
\[ d_{\inf}\left( Z_{p},\overline{\zeta}_{p}\right) >0 \]
and 
\[d_{\inf}\left( \overline{Z}_{p},\zeta_{p}\right) >0.\] 
To show that $Z_p = Z_q
$ for all $q$ sufficiently close to $p$, we will consider containment in
both directions separately.

\begin{itemize}
\item Claim 1:$\ \underset{\varepsilon>0}{\exists}\ \ \underset{q\neq p}{%
\forall}\ \ \left( \left\vert \left\vert p-q\right\vert \right\vert
_{2}<\varepsilon\ \ \Longrightarrow\ \ Z_{p}\supset Z_{q}\right) $.

If $\overline{Z}_p = \emptyset$, then $Z_p = V_\mathbb{R}(f)$ and the claim
is trivial. Thus, assume that $\overline{Z}_p\neq \emptyset$. The condition $%
Z_q\subset Z_p$ is equivalent to $d_{\inf}(\overline{Z}_p,\zeta_q) > 0$,
i.e., the finite points in $\overline{Z}_p = V_\mathbb{R}(f)\setminus Z_p$
are disjoint from $\zeta_q$ shows that $\zeta_q$ could possibly only contain
the points in $Z_p$. Since $p\in \zeta_p$, we have both $\overline{Z}_p$ and 
$\zeta_p$ are nonempty. Let $\delta = d_{\inf}\left(\overline{Z}%
_{p},\zeta_{p}\right) >0$ and take $\varepsilon>0$ that satisfies %
\eqref{eq:DistanceComponents}, i.e., 
\begin{equation}  \label{eq:DistanceSupInf}
\underset{q\neq p}{\forall}\ \ \left( \left\vert \left\vert p-q\right\vert
\right\vert _{2}<\varepsilon\ \ \Longrightarrow\ \ d_{\sup}\left( \zeta
_{q},\zeta_{p}\right) <d_{\inf}\left( \overline{Z}_{p},\zeta_{p}\right)
\right) .
\end{equation}
Fix $q\neq p$ with $\|p-q\|_2 < \varepsilon$. Since $p\in \zeta_p$ and $q\in
\zeta_q$, Lemma \ref{lem:inequality} and \eqref{eq:DistanceSupInf} yield 
\begin{equation*}
d_{\inf}(\zeta_{p},\overline{Z}_{p}) \ \ \leq \ \
d_{\sup}(\zeta_{p},\zeta_{q})\;+\;d_{\inf}(\zeta_{q},\overline{Z}_{p}) \ \ <
\ \ d_{\inf}(\overline{Z}_p,\zeta_p)\;+\;d_{\inf}(\zeta_{q},\overline{Z}%
_{p}) 
\end{equation*}
so that $d_{\inf}(\zeta_{q},\overline{Z}_{p}) > 0$.

\item Claim 2: $\ \underset{\varepsilon>0}{\exists}\ \ \underset{q\neq p}{%
\forall}\ \ \left( \left\vert \left\vert p-q\right\vert \right\vert
_{2}<\varepsilon\ \ \Longrightarrow\ \ Z_{p}\subset Z_{q}\right) $.

If $C_p=\zeta_p$, i.e., $C_p$ is connected, then $C_q$ is also connected for 
$q$ sufficiently close to $p$ by continuity so that $Z_p = Z_q = V_\mathbb{R}%
(f)$. If $\zeta_p \subsetneq C_p$, then the same holds for all $q$
sufficiently closet to $p$. In this case, one can follow essentially the
same argument as above to show $\overline{Z}_p \supset \overline{Z}_q$,
which is equivalent to $Z_p \subset Z_q$.
\end{itemize}

\bigskip

\noindent The statement~(\ref{eq:Zp=Zq}) follows from Claim 1 and Claim 2 by
choosing the $\varepsilon$ in (\ref{eq:Zp=Zq}) to be the minimum of the $%
\varepsilon$'s in Claims~1 and~2. 
Thus, we have proved the key lemma (Lemma~\ref{lem:key}).
\end{proof}

\newpage

\noindent Finally, using the key lemma (Lemma~\ref{lem:key}), we will prove   Theorem~\ref{thm:geocharacterization}.

\begin{proof}[Proof of Theorem~\ref{thm:geocharacterization}] \

\begin{description}
\item[C1:] If $d_{1}=d_{2}$, then\ \ \ $B\ \ \subset\ \ \bigcup
\limits_{r:\,C_{r}\ \text{is singular}}C_{r}.$

Assume  $d_1 = d_2$. Let $p \in B$. It suffices to show that $C_p$ is singular.
From Assumption~\ref{assumption}(\ref{lin:inf}) and $d_{1} = d_{2}$,
we see that $C_p$ is smooth at infinity.
Thus, from the key lemma (Lemma~\ref{lem:key}), 
it is immediate that $C_p$ is singular.

\item[C2:] If $d_{1}>d_{2}$, then\ \ \ $B\ \ \subset\ \ \bigcup
\limits_{r:\,C_{r}\ \text{is singular}}C_{r}\ \ \ \ \cup\ \ \ \ V_{\mathbb{R}%
}(f_{2})$.

\noindent Assume $d_1 > d_2$. Let $p \in B$ and $p\notin V_\mathbb{R}(f_2)$. It
suffices to show that $C_p$ is singular.
From Assumption~\ref{assumption}(\ref{lin:inf}) and $p \notin  V_\mathbb{R}(f_2),$
we see that $C_p$ is smooth at infinity.
Thus, from the key lemma (Lemma~\ref{lem:key}), 
it is immediate that $C_p$ is singular.
\end{description}
\end{proof}

\section{Algebraic characterization of extended boundary}
\label{sec:singPaths}

\noindent In this section, we will tackle Problem~\ref%
{problem:algcharacterization} by finding an algebraic characterization of
the extended boundary -- that is, finding polynomials that vanish on the
extended boundary. Building on the geometric characterization described in Theorem~\ref%
{thm:geocharacterization}, we start by considering $C_r$, a Newton homotopy
curve (Definition~\ref{def:NewtonHomotopyCurve}). Then, we determine a
condition on $r$ to determine if $C_r$ is singular.

Although $C_r$ is defined via a projection, it is actually algebraic as
shown in the following.

\begin{proposition}
\label{prp:definingPolyCr} If $r\in \mathbb{R}^2\setminus V_\mathbb{R}(f)$,
then $C_{r}=V_{\mathbb{R} }\left(h(x,r)\right)$ where $h(x,r)= \det%
\begin{bmatrix}
f(x) & f(r)%
\end{bmatrix}%
$.
\end{proposition}

\begin{proof}
Fix $r\in\mathbb{R}^2\setminus V_\mathbb{R}(f)$ and consider $h(x,r) = \det%
\begin{bmatrix}
f(x) & f(r)%
\end{bmatrix}%
$. In particular, since $f(r)\neq 0$, $h(x,r)$ is not the zero polynomial, so $V_\mathbb{C}(h(x,r))$ is indeed a curve in $\mathbb{C}^2$. Thus, we
aim to show that $C_r$ is precisely the real points of this complex curve.
Hence, it suffices to show 
\begin{equation*}
C_{r} = V_{\mathbb{R}}(h(x,r)).
\end{equation*}

Claim: $LHS \supset RHS$. Let $x \in RHS.$ Hence, we have $h(x,r) = 0$.
Thus, there exists $v\in\mathbb{R}^2$ with $\|v\|_2=1$ such that $f(x) v_1 +
f(r) v_2 = 0$. If $v_1=0$, then $|v_2|=1$ which implies $f(r) = 0$, which is
a contradiction. Hence, $v_1\neq 0$ so that $t = -v_2/v_1\in\mathbb{R}$ with 
$f(x)-tf(r) = 0$. Thus, by definition, $x\in LHS$.

Claim: $LHS \subset RHS$. Let $x \in LHS$. Hence, by definition, there
exists $t\in\mathbb{R}$ such that $f(x)-tf(r)=0$. Hence, $v = 
\begin{bmatrix}
1 & -t%
\end{bmatrix}%
^T$ is a nonzero null vector of $%
\begin{bmatrix}
f(x) & f(r)%
\end{bmatrix}%
$ showing that $h(x,r) = \det 
\begin{bmatrix}
f(x) & f(r)%
\end{bmatrix}
= 0$. Thus, $x \in RHS.$
\end{proof}

\bigskip

Next, we consider the \textit{singular} Newton homotopy curves. To do this,
we need a consistent way to define a Newton homotopy curve, while the
following shows that a Newton homotopy curve can be defined in many
equivalent ways.

\newpage

\begin{proposition}
If $r\in\mathbb{R}^2\setminus V_\mathbb{R}(f)$ and $s \in C_{r}\setminus V_%
\mathbb{R}(f)$, then $C_{r}=C_{s}$.
\end{proposition}

\begin{proof}
Let $r\in\mathbb{R}^2\setminus V_\mathbb{R}(f)$ and $s \in C_{r}\setminus V_%
\mathbb{R}(f)$. Then for some $t_s\in\mathbb{R}$ we have $f(s) = t_sf(r)$. 
Since $f(s)\neq 0$, we have $t_s\neq 0$. Thus $f(r)=\frac{1}{t_s}f(s)$.
Let  $x \in \mathbb{R}^2$ be arbitrary but fixed. Note
\[
x \in C_r 
\,\iff\,  \underset{t \in \mathbb{R}}\exists\,\, f(x) = t f(r) 
\,\iff\,   \underset{t \in \mathbb{R}}\exists\,\, f(x) = \frac{t}{t_s} f(s) 
\,\iff\,   \underset{t' \in \mathbb{R}}\exists\,\, f(x) = t' f(s) 
\,\iff\,   x \in C_s
\]
where $t'=\frac{t}{t_s}$.
\end{proof}

This result shows that we are able to describe a Newton homotopy curve by
selecting any point on the curve that is not contained in $V_\mathbb{R}(f)$.
By the genericity assumptions (Assumption~\ref{assumption}(\ref{lin:sing})), a singular
Newton homotopy curve has exactly one singular point and this point is not
contained in~$V_\mathbb{R}(f)$. Hence, we choose to describe each singular
Newton homotopy curve by taking the start point to be the singular point as
illustrated in the following.

\begin{center}
\begin{tabular}{ccc}
\includegraphics[width=1in]{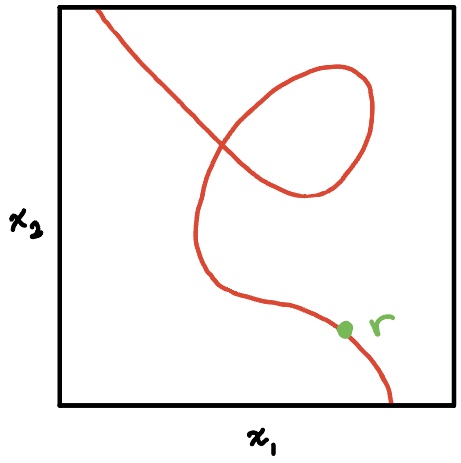}
& \hspace{1cm} & %
\includegraphics[width=1in]{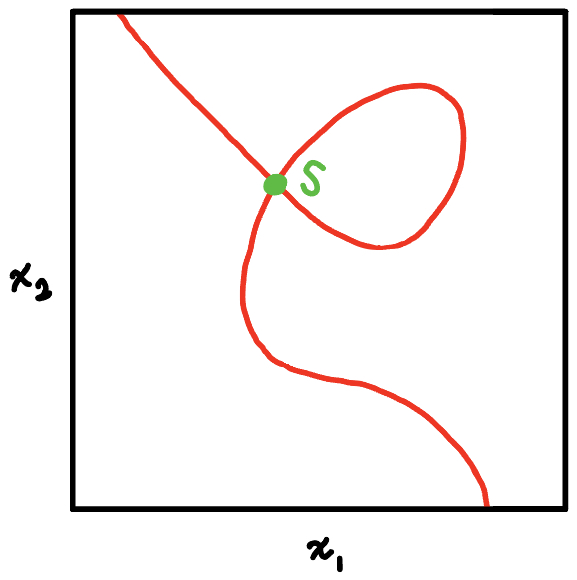}
\\ 
$C_{r}$ & \hspace{1cm} & $C_{s}$%
\end{tabular}
\end{center}

\begin{definition}[Set of singular points]
Define 
\begin{equation}  \label{def:S}
S=\left\{ r\in\mathbb{R}^{n}:\ f\left( r\right) \neq 0\ \text{and }r\text{
is a singular point of }C_{r}\right\}.
\end{equation}
\end{definition}

\begin{example}[Running Example continued]
\label{ex:singularbluepoints} For the running example (Example~\ref%
{runningExample}), a look at the first figure from Example~\ref{ex:generic_bound}
gives us an example of this set of singular points $S$ for a bivariate
system with $d_1=d_2=3$. Here, the three blue points and one green point
make up 4 singular points total with $(r_1, r_2)$ coordinates approximately $%
(0,0)$, $(-0.6393,2.1722)$, $(-2.3650, -1.0466)$, and $(0.7463, 1.7978)$,
where this last point is the green center in the figure.
\end{example}

The following lemma provides a \emph{constructible} description of $S$.

\begin{lemma}
\label{lemma:S} The set of singular points is given by $S=V_{\mathbb{R}%
}(\gamma)\setminus V_{\mathbb{R}}(f)$ where 
\begin{equation}  \label{eq:gamma}
\gamma=%
\begin{bmatrix}
\det%
\begin{bmatrix}
\partial_{x_{1}}f & f%
\end{bmatrix}
\\ 
\det%
\begin{bmatrix}
\partial_{x_{2}}f & f%
\end{bmatrix}%
\end{bmatrix}%
.
\end{equation}
\end{lemma}

\begin{proof}  Note
\begin{align*}
S &= \left\{ r\in\mathbb{R}^{n}:\ f\left( r\right) \neq 0\ \text{and }r\text{
is a singular point of }C_{r}\right\} \;\;\;\;\;\;\, \text{by \eqref{def:S}} \\
&= \{ r\in\mathbb{R}^{n} : r\text{ is a singular point of }C_{r}\} \setminus
V_{\mathbb{R}}(f) \;\;\;\;\;\;\;\;\;\;\;\;\;\;\;\;\; \text{ since } f(r)
\neq 0 \\
&= \{r\in\mathbb{R}^{n}: h(r,r) = 0 \text{ and } (\partial_{x}h)(r,r) =0\}
\setminus V_{\mathbb{R}}(f) \;\;\;\;\;\;\;\;\;\;\, \text{ since } C_{r} = V_{%
\mathbb{R}}(h(x,r)) \\
&= \{r\in\mathbb{R}^{n}: (\partial_{x}h)(r,r) =0\} \setminus V_{\mathbb{R}%
}(f) \;\;\;\;\;\;\;\;\;\;\;\;\;\;\;\;\;\;\;\;\;\;\;\;\;\;\;\;\;\;\;\;\;\;\; 
\text{ since } h(r,r) = 0 \text{ for any } r
\end{align*}
Note that $\gamma(p)$ is the transpose of the gradient of $h(x,r) = \det%
\begin{bmatrix}
f(x) & f(r)%
\end{bmatrix}%
$ with respect to $x$ evaluated at $(x,r)=(p,p)$ and $h(p,p) = 0$, using the
convention that the gradient is a row vector. In particular, 
\begin{align*}
S &= \{r\in\mathbb{R}^{n}: (\partial_{x}h^T)(r,r) =0\} \setminus V_{\mathbb{R%
}}(f) \;\;\;\;\;\; \text{ since } V_{\mathbb{R}}( (\partial_{x}h^T)(r,r)) =
V_{\mathbb{R}}( (\partial_{x}h)(r,r)) \\
&= \{r\in\mathbb{R}^{n}: \gamma(r) =0\} \setminus V_{\mathbb{R}}(f)
\;\;\;\;\;\;\;\;\;\;\;\;\;\;\;\;\;\, \text{ by the note above } \\
&= V_{\mathbb{R}}(\gamma) \setminus V_{\mathbb{R}}(f)
\end{align*}
\end{proof}

Next, we will give an \emph{algebraic} description of $S$.

\begin{lemma}
\label{lem:polysforS} Let $\gamma$ as in \eqref{eq:gamma} and $\partial_x f$
be the Jacobian matrix of $f$. Then, $S=V_{\mathbb{R}}(\gamma,\det(%
\partial_x f))$.
\end{lemma}

\begin{proof} 
Recall that, from Lemma~\ref{lemma:S}, we have that $S = V_{\mathbb{R}%
}(\gamma)\setminus V_{\mathbb{R}}(f) $. Thus, it suffices to show 
\begin{equation*}
V_{\mathbb{R}}(\gamma)\setminus V_{\mathbb{R}}(f) \,\,= \,\, V_\mathbb{R}%
(\gamma,\det(\partial_x f)).
\end{equation*}

Claim: $LHS \supset RHS.$ Let $p \in RHS.$ Then, $\gamma(p)=0$ and $%
\det(\partial_{x}f)(p)=0.$ By the genericity assumptions (Assumption~\ref{assumption}(\ref{com-int})), $f(p)\neq0.$ Hence, $p \in LHS.$

Claim: $LHS \subset RHS.$ Let $p \in LHS$. Then $\gamma(p)=0$ and $f(p) \neq0
$. By writing 
\begin{equation*}
h(x,r) = \det%
\begin{bmatrix}
f(x) & f(r)%
\end{bmatrix}
= f_2(r) f_1(x) - f_1(r) f_2(x),
\end{equation*}
the vanishing of the transpose of the gradient of~$h(x,r)$ with respect to $x
$ evaluated at $(p,p)$ is equivalent to $%
\begin{bmatrix}
f_2(p) & -f_1(p)%
\end{bmatrix}%
$ being a left null vector of~$\partial_x f(p)$. Since $LHS\subset\mathbb{R}%
^2\setminus V_\mathbb{R}(f)$, such a null vector is nonzero. Thus, $%
\det(\partial_x f)(p)=0.$ Hence, $p \in RHS$.
\end{proof}

Finally we are ready to state the main result of this section.

\begin{theorem}[Algebraic characterization of the extended boundary]
\label{thm:DefiningPolynomialsBdry} For $S$ as in \eqref{def:S}, define $%
\mathcal{H}(x) = \prod_{r \in S} h(x,r)$ and $\mathcal{H}_2 = \mathcal{H}%
\cdot f_2$. Then, we have the following.

\begin{description}
\item[C1:] If $d_{1}=d_{2}$, then\ \ \ $\overline{B}\ \ \subset\ \ V_\mathbb{%
R}(\mathcal{H})$.

\item[C2:] If $d_{1}>d_{2}$, then\ \ \ $\overline{B}\ \ \subset\ \ V_\mathbb{%
R}(\mathcal{H}_2)$.
\end{description}
\end{theorem}
\begin{proof}
Immediate from Proposition~\ref{prp:definingPolyCr}, Lemmas~%
\ref{lemma:S} and~\ref{lem:polysforS}, and Theorem~\ref%
{thm:geocharacterization}.
\end{proof}

\begin{example}[Running Example continued]
\label{ex:gamma} For the running example (Example~\ref{runningExample})
in which \mbox{$d_1=d_2=3$}, we can define the system $\gamma$ as follows: 
\begin{align*}
\gamma&=%
\begin{bmatrix}
\det[\partial_{x_1}f & f] \\ 
\det[\partial_{x_2}f & f]%
\end{bmatrix}%
=%
\begin{bmatrix}
\det%
\begin{bmatrix}
3x_1^2 - x_2 & x_1^3 - x_1x_2 - x_2^2 + 1 \\ 
x_2 - 2x_1 & - x_1^2 + x_1x_2 + x_2^3 + 1%
\end{bmatrix}
\\ 
\det 
\begin{bmatrix}
- x_1 - 2x_2 & x_1^3 - x_1x_2 - x_2^2 + 1 \\ 
3x_2^2 + x_1 & - x_1^2 + x_1x_2 + x_2^3 + 1%
\end{bmatrix}%
\end{bmatrix}
\\
&= 
\begin{bmatrix}
- x_1^4 + 2x_1^3x_2 + 3x_1^2x_2^3 - x_1^2x_2 + 3x_1^2 - 2x_1x_2^2 + 2x_1 -
x_2^4 + x_2^3 - 2x_2 \\ 
- x_1^4 - 3x_1^3x_2^2 + x_1^3 + 2x_1^2x_2 + 2x_1x_2^3 - x_1x_2^2 - 2x_1 +
x_2^4 - 3x_2^2 - 2x_2%
\end{bmatrix}%
\end{align*}
This system has 21 solutions of
which 7 are real and 3 of these satisfy $f$.
Thus, we are
left with the four singular points as previously shown in Example~\ref%
{ex:singularbluepoints}.
\end{example}

\section{Local structure of  extended boundary at singularity}

\label{sec:TangentCone}

We now consider Problem~\ref{problem:localStructure} for characterizing 
the local structure of the extended boundary at a~singularity. 

\medskip

\begin{lemma}
\label{lem:tangentgamma} Let $r \in S$. Let $T_r$ be the tangent cone of $%
C_{r}$ at $r$. Then, we have

\begin{enumerate}
\item $T_{r} =\{v \in \mathbb{C}^{2}: v^T(\partial_{x}\gamma)(r)v=0\} $

\item If $r$ is not an isolated point in $C_r$, then $T_r$ consists of two
real lines.
\end{enumerate}
\end{lemma}

\begin{proof}
\ 

\begin{enumerate}
\item As in the proof of Lemma~\ref{lemma:S}, $\gamma(p)$ is the transpose
of the gradient of $h(x,r)$ with respect to $x$ evaluated at $(x,r) = (p,p)$%
, using the convention that the gradient is a row vector. Hence,~$\partial_x
\gamma$ is the corresponding Hessian matrix of $h(x,r)$ since the Hessian of
a polynomial is a symmetric matrix. Thus, this claim follows from the
genericity assumptions (Assumption~\ref{assumption}(\ref{lin:sing})) in that~$r$
is an ordinary double point.

\item As $r$ is an ordinary double point, $(\partial_x \gamma)(r)$ is a full
rank matrix. Hence, $T_r$ consists of two lines, which either are both real
or complex conjugates of each other. When $r$ is not an isolated point in $%
C_r$, at least one of the lines in $T_r$ must be real since nonisolated
implies the existence of real secant lines arbitrarily close to a line in $%
T_r$. Hence, it follows that both of the lines in $T_r$ must be real when $r$
is not an isolated point in $C_r$.
\end{enumerate}
\end{proof}

\renewcommand{\arraystretch}{1.7}
\begin{example}[Running Example continued]
\label{ex:tangentconecalc} Let us consider the 
local structure for the running example (Example~\ref%
{runningExample}). From Example~\ref{ex:gamma}, we have obtained 
\begin{equation*}
\gamma = \left[
\begin{array}[c]{l}
- x_1^4 + 2x_1^3x_2 + 3x_1^2x_2^3 - x_1^2x_2 + 3x_1^2 - 2x_1x_2^2 + 2x_1 -
x_2^4 + x_2^3 - 2x_2 \\ 
- x_1^4 - 3x_1^3x_2^2 + x_1^3 + 2x_1^2x_2 + 2x_1x_2^3 - x_1x_2^2 - 2x_1 +
x_2^4 - 3x_2^2 - 2x_2%
\end{array} 
\right]
\end{equation*}
\renewcommand{\arraystretch}{1.3}
We will compute the tangent cone of $C_{(0,0)}$ at $(0,0)$ using Lemma~\ref%
{lem:tangentgamma}. Note 
\begin{align*}
T_{(0,0)}&=\{v : v^T(\partial_{x}\gamma)(0,0)v=0\} \;\;\;\;\;\;\;\;\;\;\;\;
\;\;\;\;\;\;\;\;\;\;\text{ from Lemma~\ref{lem:tangentgamma}} \\
&= \left\{ v :v^T%
\begin{bmatrix}
2 & -2 \\ 
-2 & -2%
\end{bmatrix}
v=0 \right\} \;\;\;\;\;\;\;\;\;\;\; \;\;\;\;\;\;\;\,\text{ by evaluating 
} (\partial_{x}\gamma)(0,0) \\
&= \text{span}%
\begin{bmatrix}
1+\sqrt{2} \\ 
1%
\end{bmatrix}%
\;\; \cup \;\;\text{span}%
\begin{bmatrix}
1-\sqrt{2} \\ 
1%
\end{bmatrix}%
\;\;\;\;\; \, \text{ by solving for } v.
\end{align*}
Note that $(0,0)$ is not an isolated point. Thus, as Lemma~\ref%
{lem:tangentgamma} claims, $T_{(0,0)}$ consists of two real lines.
\end{example}
\renewcommand{\arraystretch}{1.0}

\newpage

\begin{theorem}[Local structure of extended boundary at a singularity]
\label{thm:tangentcone} If $r\in S$ and $\gamma$ as in \eqref{eq:gamma},
then the tangent cone $T_r$ of $C_r$ at $r$ is equal to the union of the
eigenspaces of~$(\partial_{x}P \gamma)(r)$ where 
\begin{equation}  \label{P}
P = 
\begin{bmatrix}
0 & -1 \\ 
1 & 0%
\end{bmatrix}
\text{, \;\;a rotation by } 90^{\circ}.
\end{equation}
\end{theorem}

\begin{proof}
Note 
\begin{align}
T_{r} &=\{v \in \mathbb{C}^{2}: v^T(\partial_{x}\gamma)(r)v=0\}
\;\;\;\;\;\;\;\;\;\;\;\; \;\;\;\;\;\;\;\;\;\;\;\;\;\;\;\;\;\;\, \text{ from
Lemma~\ref{lem:tangentgamma}}  \notag \\
&= \{v \in \mathbb{C}^{2}: v^T(P^{T}P)(\partial_{x}\gamma)(r)v=0\} \;\;
\;\;\;\;\;\;\;\;\;\;\;\;\;\;\;\;\;\; \text{ since $P$ is orthonormal}  \notag
\\
&= \{v \in \mathbb{C}^{2}: (Pv)^{T}(\partial_{x}P\gamma)(r)v=0\} \;\;\;\;
\;\;\;\;\;\;\;\;\;\;\;\;\;\;\;\;\;\; \text{ by rearranging}  \notag \\
&= \{v \in \mathbb{C}^{2}: (Pv)^{T}Mv=0\}
\;\;\;\;\;\;\;\;\;\;\;\;\;\;\;\;\;\;\;\;\;\;\;\;\;\;\;\;\;\;\;\;\;\; \text{
by using a shorthand } M = (\partial_{x}P\gamma)(r)  \notag \\
&= \{v \in \mathbb{C}^{2}: \det 
\begin{bmatrix}
v & Mv%
\end{bmatrix}%
=0 \} \;\;\;\;\;\;\;\;\;\;\;\; \;\;\;\;\;\;\;\;\;\;\;\;\;\;\;\;\;\;\text{
from }\; \forall w \; (Pv)^{T}w= \det 
\begin{bmatrix}
v & w%
\end{bmatrix}
\notag \\
&= \{v \in \mathbb{C}^{2}: v \text{ and } Mv \;\text{ are linearly dependent}
\} \;\;\; \text{ from a basic property of determinant}  \label{ld1} \\
&= \{ v \in \mathbb{C}^{2}: \exists_{\lambda} Mv = \lambda v \}
\;\;\;\;\;\;\;\;\;\;\;\;\;\;\;\;\;\;\;\; \;\;\;\;\;\;\;\;\;\;\;\;\;\;\;\;\;\;%
\text{ see below}  \label{ld2} \\
&= \text{union of eigenspaces of } (\partial_{x}P\gamma)(r).  \notag
\end{align}
We remark on the claimed equality between the two sets (\ref{ld1}) and (\ref%
{ld2}).  The inclusion (\ref{ld1}) $\supset$ (\ref{ld2}) is immediate from
the definition of linear dependence.  The inclusion (\ref{ld1}) $\subset$ (%
\ref{ld2}) is subtle but still follows from the special relationship between
the two vectors $v$ and $Mv$.
\end{proof}

\begin{example}[Running Example continued]
Let us consider 
computing the tangent cone of $C_{(0,0)}$ at $(0,0)$ 
using Theorem~\ref{thm:tangentcone}
for the running example (Example~\ref{runningExample}). 
Note 
\begin{align*}
(\partial_xP\gamma)(0,0) &= P(\partial_x\gamma)(0,0)
\;\;\;\;\;\;\;\;\;\;\;\;\;\;\;\text{ since } P \text{ is a constant matrix}
\\
&=%
\begin{bmatrix}
0 & -1 \\ 
1 & 0%
\end{bmatrix}
\begin{bmatrix}
2 & -2 \\ 
-2 & -2%
\end{bmatrix}%
\;\;\; \text{ from \ref{P} and by evaluating } (\partial_{x}\gamma)(0,0) \\
&= 
\begin{bmatrix}
2 & 2 \\ 
2 & -2%
\end{bmatrix}%
\end{align*}
One can directly find that the above matrix has the following eigenspaces: 
\begin{equation*}
\text{span}%
\begin{bmatrix}
1+\sqrt{2} \\ 
1%
\end{bmatrix}
\;\;\; \text{ and } \;\;\; \text{span}%
\begin{bmatrix}
1-\sqrt{2} \\ 
1%
\end{bmatrix}%
.
\end{equation*}
Thus, from Theorem~\ref{thm:tangentcone}, we conclude that the tangent cone $%
T_{(0,0)}$ is the union of these two eigenspaces. Note that this is
identical to the tangent cone computed in Example~\ref{ex:tangentconecalc}.
\end{example}

\section{Upper bound on the number of cells}
\label{sec:NumberConnected}

Finally, we consider Problem~\ref{prob:NumCells} regarding bounding the
number of connected components of $\mathbb{R}^2\setminus \overline{B}$. For
a given $f$, one could directly compute the number of cells exactly by
applying a known method such as~\cite{BabyGiantSemiAlg, Collins75,
SmoothConnectivity, Hong2010}. However, the aim here is to determine an
upper bound in terms of the degrees of the polynomials in $f$.
Since $\mathcal{H}$ in Theorem~\ref{thm:DefiningPolynomialsBdry} is defined
in terms of $S$, we first find an upper bound on $\# S$
and then derive bounds separately for 
when the degrees are equal 
and when the degrees are different.

\newpage

\subsection{Upper bound on number of singular points}

\begin{theorem}[Upper bound on number of singular points]
\label{thm:num_S} We have

\begin{description}
\item[C1:] If $d_{1}=d_{2}$, then\;\; $\# S \;\;\leq\;\; 3(d_1-1)^2$.~\footnote{Note that the case
when $d_1=d_2$ corresponds with the degree of the set of singular curves in $%
\mathbb{C}^2$ of degree $d_1$, which is a special case of a Boole's formula,
e.g., see \cite[Chap.~9]{GKZ}.
}
\item[C2:] If $d_{1}>d_{2}$, then\;\; $\#S \;\;\leq\;\; (d_1+d_2-1)^2 - d_1d_2$.
\end{description}
\end{theorem}

\begin{proof}
We begin with a few notations.
For a polynomial system $F$, let $\mathcal{N}_\mathbb{C}(F)$ denote the number
of solutions of $F$ in $\mathbb{C}^2$ counting multiplicity and $\mathcal{N}%
_\infty(F)$ denote the same at infinity.

From Lemma~\ref{lemma:S}, we have
$S \;= \; V_\mathbb{R}(\gamma)\setminus V_\mathbb{R}(f).$
Since $V_\mathbb{C}(f)\subset V_\mathbb{C}(\gamma)$ and 
$\# V_\mathbb{C}(\gamma) \leq \mathcal{N}_\mathbb{C}(\gamma)$, we~have
\begin{equation*}
\# S 
\;\;\;=\;\;\; \#V_{\mathbb{R}}(\gamma) - \#V_{\mathbb{R}}(f) 
\;\;\;\leq\;\;\; \#V_{\mathbb{C}}(\gamma) - \#V_{\mathbb{C}}(f)  
\;\;\;\leq\;\;\; \mathcal{N}_\mathbb{C}(\gamma) - \#V_{\mathbb{C}}(f).
\end{equation*}
From the genericity condition (Assumption~\ref{assumption}(\ref{com-int})), we have $\#V_{\mathbb{C}}(f)  = d_1d_2$. Thus,
\begin{equation*}
\# S \;\;\leq\;\; \mathcal{N}_\mathbb{C}(\gamma) -d_1d_2.
\end{equation*}
Since $\gamma$ consists of two polynomials of degree at most $d_1+d_2-1$, B\'ezout's
bound provides 
\begin{equation}
\# S \;\;\leq\;\;  (d_1+d_2-1)^2 - \mathcal{N}_\infty(\gamma) - d_1d_2.   \label{Sbound}
\end{equation}

It remains to estimate $\mathcal{N}_\infty(\gamma)$. To that
end, we compute the number of solutions in~$\mathbb{P}^{1}$ of the highest
degree homogeneous part of $\gamma$, denoted $\gamma^{\ast}$. Since $%
{}^{\ast}$ is a homomorphism, \eqref{eq:gamma}~yields 
\begin{equation*}
\gamma^{\ast}=\left[ 
\begin{array}{c}
\det\left[ 
\begin{array}{ll}
\partial_{x_{1}}f^{\ast} & f^{\ast}%
\end{array}
\right] \\ 
\det\left[ 
\begin{array}{ll}
\partial_{x_{2}}f^{\ast} & f^{\ast}%
\end{array}
\right]%
\end{array}
\right] . 
\end{equation*}
Euler's theorem on homogeneous functions provides that 
\begin{equation}  \label{eq:grelation}
x_{1}\gamma_{1}^{\ast} + x_{2}\gamma_{2}^{\ast} = (d_{1}-d_{2})
f_{1}^{\ast}f_{2}^{\ast}.
\end{equation}
We consider the two cases.

\begin{itemize}
\item Case: $d_{1}=d_{2}$. We claim that $\mathcal{N}_\infty(\gamma) = 2(d_1-1)$.
By~\eqref{eq:grelation}, $x_{1}\gamma_{1}^{\ast} = -x_{2}\gamma_{2}^{\ast}$ so that $%
x_{1}$ is a factor of $\gamma_2^{\ast}$. Hence, $\phi = \gamma_2^{\ast}/x_1$ is a
polynomial of degree $d_1+d_2-2=2(d_1-1)$ such that $\gamma_1^{\ast} = -x_2 \phi$
and $\gamma_2^{\ast} = x_1 \phi$. Therefore, by the Fundamental Theorem of
Algebra, $\mathcal{N}_\infty(\gamma) = 2(d_1-1)$. From the claim and \eqref{Sbound}, we finally have
\begin{equation*}
\# S \;\;\leq\;\; (d_1+d_2-1)^2  - 2(d_1-1) - d_1d_2 \;\;=\;\; (2d_1-1)^2  - 2(d_1-1) - d_1^2
\;\;=\;\; 3(d_1-1)^2.
\end{equation*}

\item Case: $d_{1} > d_{2}$. We claim that $\mathcal{N}_\infty(\gamma) = 0$.
We will prove it by contradiction. Assume that $\gamma$ has a solution at infinity. By~%
\eqref{eq:grelation}, either $f_{1}^{\ast}(x) = 0$ or $f_{2}^{\ast}(x)=0$,
but not both by genericity assumptions (Assumption~\ref{assumption}(\ref{com-int})) and any
such solution at infinity is nonsingular.
Assume that $f_{1}^{\ast}(x) = 0$ so that $f_{2}^{\ast}(x)\neq0$. Then, 
$\
0=\gamma^{\ast}(x) = f_{2}^{\ast}(x) \partial_x f_{1}^{\ast}(x) 
$
which is a contradiction since $f_{2}^{\ast}(x)\neq0$ and $\partial_x
f_{1}^{\ast}(x) \neq0$. Repeating a similar argument assuming that $%
f_{2}^{\ast}(x) = 0$ also yields a contradiction completing the proof.
From the claim and \eqref{Sbound}, we finally have
\begin{equation*}
\# S \;\;\leq\;\; (d_1+d_2-1)^2 - 0 - d_1d_2 \;\;=\;\; (d_1 +d_2-1)^2 - d_1d_2.
\end{equation*}
\end{itemize}
\end{proof}

\medskip

In the following, we will give four examples with degrees  $(2,1), (2,2), (3,1)$ and $(3,2)$ respectively
where the bounds in Theorem~\ref{thm:num_S} are exact. 
We leave it as an open problem to check whether for all values of $d_1\geq d_2$
with $d_1\geq 2$, there exists an example where the bounds are exact.

\begin{example}
\label{ex:21caseS} Consider the example: 
\begin{equation*}
\begin{split}
f_{1} & = x_1^2 + 3 x_1 x_2 - 5 x_2^2 + 9 x_1 - 2 x_2 + 40 \\
f_2 & = 5 x_1 - 2 x_2 + 1.
\end{split}%
\end{equation*}
\pichskip{0pt} 
\parpic[r][t]{
\begin{tikzpicture}[scale=0.15]
\draw  (-10,-6.46)     --  (10,10.3);
\draw  (-10,-11.11)     --  (10,5.668);
\draw  (-10,5.2)     --  (10,0.439);
\draw  (-9.95,-2.469)     --  (9.95,-7.2);
\draw (-4.2, -9.9) -- (3.7,9.93584);
\filldraw[red]  (-1.94,-4.37)   circle (6pt); 
\filldraw[red]  (0.85,2.63)     circle (6pt); 
\filldraw[blue] (-6.27,-3.34)   circle (6pt);
\filldraw[blue] (5.17,1.59)     circle (6pt); 
\end{tikzpicture}
}
\noindent Note that $d_{1}=2$ and $d_{2}=1$. From Theorem~\ref{thm:num_S}, we have 
\begin{equation*}
\#S \;\;\leq\;\; (d_1+d_2-1)^2-d_1d_2 \;=\; 2. 
\end{equation*}
Look at the figure on the right. The red dots are the real solutions of $f$
and the blue dots are elements of $S$. Note that there are \emph{two} blue dots,
that is, $\#S=2$. 
Thus, for this example, the upper bound for $\#S$ is \emph{exact}.
\end{example}

\begin{example}
\label{ex:22caseS} 
Consider the example (which was studied in \cite%
{Newton22case}).
\begin{align*}
f_{1} & =-11x_{1}^{2}+8x_{1}x_{2}-2x_{2}^{2}+63x_{1}-112 \\
f_{2} & = -8x_{1}^{2}+5x_{1}x_{2}-8x_{2}^{2}+54x_{1}+54x_{2}-196.
\end{align*}
\pichskip{0pt} 
\parpic[r][t]{
\begin{tikzpicture}[scale=0.18]
\draw  (-5,12)     --  (12,-5);
\draw  (2,12) --  (2,-6);
\draw  (12,11.33)   --  (-1,-6);
\draw  (12,8)--  (-6,-1);
\draw  (12,6.66)    --  (-6,3.66);
\draw  (12,-0.33)  --  (-6,5.66);
\filldraw[red]  (2,5)       circle (4pt);
\filldraw[red]  (8,6)       circle (4pt);
\filldraw[red]  (5,2)       circle (4pt);
\filldraw[red]  (2,3)       circle (4pt);
\filldraw[blue] (3.33,3.66) circle (4pt);
\filldraw[blue] (2,-2)      circle (4pt);
\filldraw[blue] (-2,4.33)   circle (4pt);
\end{tikzpicture}
}
\noindent Note that $d_{1}=2$ and $d_{2}=2$. From Theorem~\ref{thm:num_S}, we have 
\begin{equation*}
\#S \;\;\leq\;\; 3(d_1-1)^2 \;=\; 3. 
\end{equation*}
In the figure on the right, the red dots are the real solutions of $f$
and the blue dots are elements of $S$. Note that there are \emph{three} blue dots,
that is, $\#S=3$.
Thus, for this example, the upper bound for $\#S$ is \emph{exact}.
\end{example}

\begin{example}
\label{ex:31caseS}
Consider the example:  
\begin{equation*}
\begin{split}
f_{1} & =37x_{1}^{3} - 39x_{1}^{2}x_{2} - 27x_{1}x_{2}^{2} - 6x_{2}^{3} -
4x_{1}^{2} + 95 x_{1}x_{2} + 55x_{2}^{2} -68x_{1} +97x_{2} -64 \\
f_{2} & = 18x_{1} + 9x_{2} -94.
\end{split}%
\end{equation*}
\pichskip{0pt} 
\parpic[r][t]{
\includegraphics[width=0.25\textwidth]{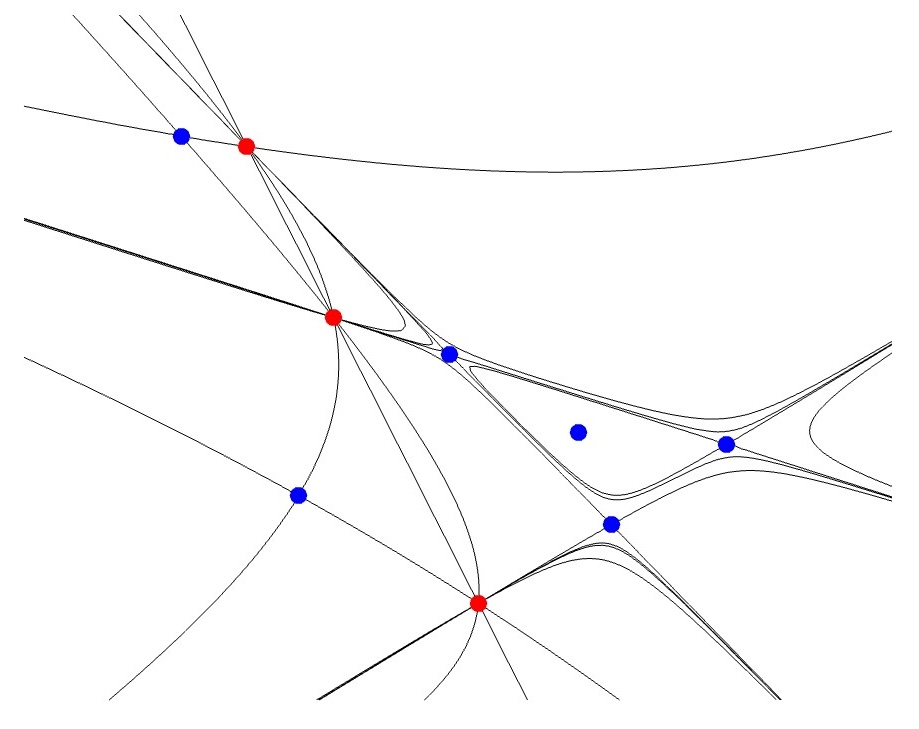}
}
\noindent Note that $d_{1}=3$ and $d_{2}=1$. 
From Theorem~\ref{thm:num_S}, we have 
\begin{equation*}
\#S \;\;\leq\;\; (d_1+d_2-1)^2-d_1d_2 \;=\; 6. 
\end{equation*}
Looking at the figure on the right, the red dots are the real solutions of $f$
and the blue dots are elements of $S$. Note that there are \emph{six} blue dots,
that is, $\#S=6$.
Thus, for this example, the upper bound for $\#S$ is \emph{exact}.
\end{example}


\begin{example}
\label{ex:32caseS}
Consider the example:
\begin{equation*}
\begin{split}
f_1&= 30x_1^2x_2-50x_1x_2^2+36x_1^2+49x_1x_2-2x_2^2+50x_1 -10x_2 +37\\
f_2&=23x_1^2+37x_1x_2+14x_2^2-28x_1-33x_2-8
\end{split}%
\end{equation*}
\pichskip{0pt} 
\parpic[r][t]{
\includegraphics[width=0.5\textwidth]{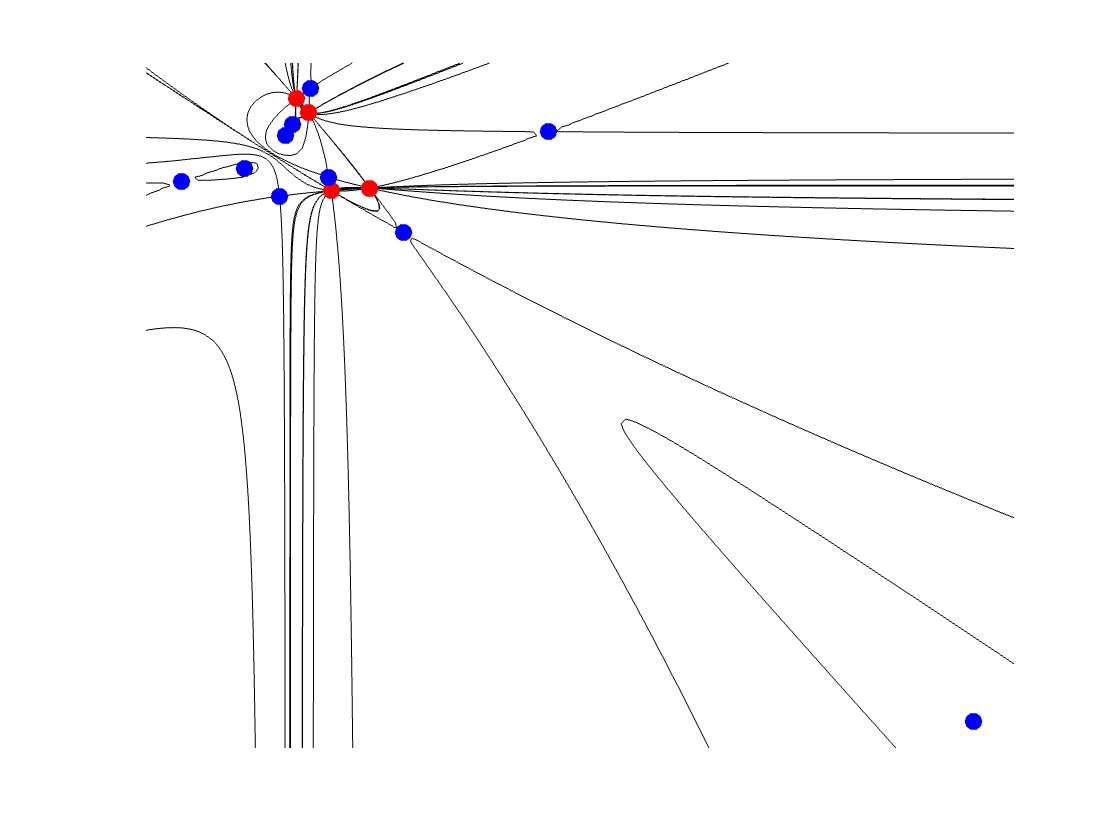}
}
\noindent 
Note that $d_1 = 3$ and $d_2 = 2$. From Theorem~\ref{thm:num_S}, we have
\begin{equation*}
\#S \;\;\leq\;\; (d_1+d_2-1)^2-d_1d_2 \;=\; 10. 
\end{equation*}
In the figure on the right, the red dots are the real solutions of~$f$
and the blue dots are elements of $S$. Note that there are \emph{ten} blue dots,
that is, $\#S=10$.
Thus, for this example, the upper bound for $\#S$ is \emph{exact}.
\end{example}

\vspace{1em}

\subsection{Upper bound on number of cells when degrees are equal}
\label{sec:DegreesEqual}

\begin{theorem}[Upper bound on number of cells when $d_1=d_2$ ]
\label{thm:num_comp_equal} When $d_1=d_2$, we have
\begin{equation}  \label{eq:Nfequal}
\#\mathrm{cells} \;\;\leq\;\; \#S(2d_{1}^2-1)-d_{1}^2+1 \;\;\leq\;\;
6d_1^4-12d_1^3+2d_1^2+6d_1-2.
\end{equation}
\end{theorem}

\begin{proof}
Let $d_1 = d_2$.
The upper bound in \eqref{eq:Nfequal}
arises from a bound on the number of routing 
points~\cite{SmoothConnectivity,Hong2010}, which
provides an upper bound on the number of cells since each cell 
must contain at
least one routing point~\cite[Prop.~3.2]{SmoothConnectivity}.

Let
\begin{eqnarray} 
\mathcal{H}(x) &=& \prod_{r\in S}h(x,r)  \text{\;\;\;\;  as in Theorem~\ref%
{thm:DefiningPolynomialsBdry}} \\
\rho &=& \prod_{r\in S} f_2(r)  \label{eq:rho} \\
p(x) &=& \frac{\mathcal{H}(x)}{\rho} = \prod_{r\in S} \left( f_1(x) - \frac{%
f_1(r)}{f_2(r)} f_2(x)\right).  \label{eq:PolyP}
\end{eqnarray}
\noindent {\sf Claim}: $p(x)$ is well-defined. 
It suffices to show that $\rho \neq 0$.
Let $r\in V_\mathbb{R}(f_2)$ be arbitrary.
It suffices to show that $r \notin S$.
\begin{itemize}
\item  When $r \in V_\mathbb{R}(f_1)$:
     From~(\ref{def:S}), we have $r  \notin S$.
\item When $r \notin V_\mathbb{R}(f_1)$:
      From Proposition~\ref{prp:definingPolyCr}, we immediately have  $C_r = V_\mathbb{R}(f_2)$.
From the genericity assumption (Assumption~\ref{assumption}(\ref{smooth})), we have that $V_\mathbb{C}(f_2)$ is smooth.
Hence $C_r$ is smooth. Thus $r \notin S$.
\end{itemize}
The claim has been proved.

\bigskip

\noindent 
It is obvious that $\deg p \;\leq\; (\# S) d_1$. 
Due to the genericity assumption (Assumption~\ref{assumption}(\ref{lin:deg})), 
we~have 
\begin{equation}
\deg p = (\# S) d_1.
\end{equation}
Let  
\begin{equation}
e = \lceil (\deg p + 1)/2 \rceil.  \label{eq:e}
\end{equation} 
Note that $2e > \deg p$. 
For a general $c\in\mathbb{R}^2$, we consider the routing
function 
\begin{equation*}
u(x) = \frac{p(x)}{q(x)^e}
\end{equation*}
where  
\begin{equation}
    q(x) = 1 + (x_1-c_1)^2 + (x_2-c_2)^2.  \label{eq:PolyQ}
\end{equation} 
The set of routing points is $V_\mathbb{R}(\partial_x u)\setminus V_\mathbb{R%
}(p)$ where 
\begin{equation}  \label{eq:dxU}
\partial_x u = \frac{\partial_x p}{q^{e}} - e \frac{(\partial_x q) p}{q^{e+1}%
} = \frac{(\partial_x p) q-e (\partial_x q)p}{q^{e+1}}=0.
\end{equation}
Since the numerator of $\partial_x u$ has degree $%
(\#S) d_1+1$, B\'ezout's bound provides that the number of routing points is
at most $((\# S)d_1+1)^2$. The following sharpens this bound by identifying
several cases that can not yield routing points.

\begin{enumerate} [align=parleft, leftmargin=5em,  labelwidth=3em]

\item[$S$:] For every $r\in S$, $r$ is a singular point of $h(x,r)$ and thus
the numerator of $\partial_x u$ will vanish at~$r$. Hence, we can reduce the
bound by $\# S$.

\item[$p,q$:] Since $q$ is the denominator of $\partial_x u$ and $V_\mathbb{R%
}(q) = \emptyset$, we can ignore any points arising from the vanishing of $q$%
. Since the simultaneous vanishing of $p$ and $q$ contributes $2(\#S)d_1$,
we can reduce the bound by $2(\#S)d_1$.

\item[$V_\mathbb{C}(f)$:] For every $r\in S$, we know that $h(x,r)$ vanishes
on $V_\mathbb{C}(f)$, which consists of $d_1^2$ distinct points by
genericity assumptions (Assumption~\ref{assumption}(\ref{com-int})). Since $p$ arises as a
product of $\# S$ bivariate polynomials, each point in $V_\mathbb{C}(f)$ has
multiplicity $(\#S-1)^{2}$ with respect to the numerator of $\partial_x u$.
Thus, we can reduce the bound by $d_{1}^2(\#S-1)^{2}$.
\end{enumerate}

\noindent Since each of the cases above are disjoint, the number of
routing points is at most 
\begin{equation*}
((\# S)d_1+1)^2 - \# S - 2(\# S) d_1 - d_1^2(\# S - 1)^2 = \# S(2d_1^2-1) -
d_1^2 + 1. 
\end{equation*}
The last inequality in \eqref{eq:Nfequal} follows from $\# S \leq 3(d_1-1)^2$
by Theorem~\ref{thm:num_S}.
\end{proof}

\medskip

In the following, we give an example where the bound in Theorem~\ref{thm:num_comp_equal} is exact.
\begin{example}
Consider the example (which was studied in \cite{Newton22case} and also in Example~\ref{ex:22caseS} ):
\begin{align*}
f_{1} & =-11x_{1}^{2}+8x_{1}x_{2}-2x_{2}^{2}+63x_{1}-112 \\
f_{2} & = -8x_{1}^{2}+5x_{1}x_{2}-8x_{2}^{2}+54x_{1}+54x_{2}-196.
\end{align*}
\pichskip{0pt} 
\parpic[r][t]{
\begin{tikzpicture}[scale=0.2]
\tikzstyle{every node}=[font=\tiny]
\draw  (-5,12)     --  (12,-5);
\draw  (2,12) --  (2,-6);
\draw  (12,11.33)   --  (-1,-6);
\draw  (12,8)--  (-6,-1);
\draw  (12,6.66)    --  (-6,3.66);
\draw  (12,-0.33)  --  (-6,5.66);
\filldraw[red]  (2,5)       circle (4pt);
\filldraw[red]  (8,6)       circle (4pt);
\filldraw[red]  (5,2)       circle (4pt);
\filldraw[red]  (2,3)       circle (4pt);
\filldraw[blue] (3.33,3.66) circle (4pt);
\filldraw[blue] (2,-2)      circle (4pt);
\filldraw[blue] (-2,4.33)   circle (4pt);
\node (1) at (3.44,2.89) {$1$};
\node (2) at (2.44,3.89) {$2$};
\node (3) at (0.67,4.11) {$3$};
\node (4) at (-4.67,4.56) {$4$};
\node (5) at (-3.14,7.47) {$5$};
\node (6) at (-0.33,9.67) {$6$};
\node (7) at (6.08,9.06) {$7$};
\node (8) at (4.44,4.89) {$8$};
\node (9) at (10.67,8.44) {$9$};
\node (10) at (5.44,3.89) {$10$};
\node (11) at (10.67,6.89) {$11$};
\node (12) at (9.37,3.25) {$12$};
\node (13) at (9.67,-1.11) {$13$};
\node (14) at (5.61,-3.31) {$14$};
\node (15) at (3,1) {$15$};
\node (16) at (1,-4.67) {$16$};
\node (17) at (-1.93,-2.13) {$17$};
\node (18) at (-3.1,2.2) {$18$};
\end{tikzpicture}
}
\noindent Note that $d_1=2$ and $d_2 = 2$. 
From Theorem~\ref{thm:num_comp_equal}, we have
\[
   \# \text{cells} \;\;\leq\;\;  6d_{1}^{4}-12d_{1}^{3}+2d_{1}^{2}+6d_{1}-2 \;=\; 18.
\]
Consider the figure on the right. The red dots are the real solutions of $f$ and
the blue dots are elements of $S$. Note that there are 18 cells.
Thus, for this example, the upper bound for $\#$ cells  is \emph{exact}.
\end{example}

\subsection{Upper bound on number of cells when degrees are different}

\label{sec:DegreesNotEqual}
Before we state the main result, we first briefly go over the overall strategy. 
Let $p$ be as in \eqref{eq:PolyP};  we consider $p_2 = p\cdot f_2$ as one
would expect from Theorem~\ref{thm:DefiningPolynomialsBdry}. The proof of
Theorem~\ref{thm:num_comp_equal} identified three cases that are not routing
points. With appropriate modification, these extend to the case when $d_1 >
d_2$. However, there is one additional type of points to remove: solutions
at infinity. The reason is that for every $r$ with $f_2(r)\neq 0$, $h(x,r)$
now defines the same~$d_1$ points at infinity. In fact, the terms of degree 
$d_1, d_1-1, \dots, d_1-d_2+1$ are the same, i.e., the top $d_1-d_2$
degrees. Therefore, multiplying $\# S$ polynomials together with the same $%
d_1-d_2$ leading degrees yields a scheme at infinity that can be ignored
when bounding the number of cells.

\begin{definition}
\label{def:inf_complete} 
We say that $f$ is \textit{infinity complete} if
there exists $c\in\mathbb{R}^2$ such that 
\begin{equation*}
(\partial_x p_2) q - e (\partial_x q) p_2=0 
\end{equation*}
has only finitely many solutions and the scheme at infinity has degree $%
\#S(\#S-1)d_{1}(d_{1}-d_{2})$, 
where  $p, q$ and $e$ are as defined in (\ref{eq:PolyP}),(\ref{eq:PolyQ}) and (\ref{eq:e}), respectively.
\end{definition}

\begin{remark}
When $d_{1}=d_{2}$, the systems $f$ under consideration are infinity
complete with no corresponding solutions at infinity due to the genericity
assumptions (Assumption~\ref{assumption}).
\end{remark}

Note that infinity completeness is an open condition on the system $f$. We
selected random cases with $1\leq d_2 < d_1 \leq 3$ with $\# S \leq 4$ and
all were verified to be infinity complete using \texttt{Macaulay2}~\cite%
{Macaulay2} and \texttt{Bertini}~\cite{BHSW06}. This suggests the following.

\begin{conjecture}
\label{conj:Infinity} Almost all $f$'s are infinity complete.
\end{conjecture}

Now we are ready to state the main result.

\begin{theorem}[Upper bound on number of cells when $d_1 >d_2$ ]
\label{thm:num_comp} When $d_1 > d_2$, we have
\begin{equation*}
\# \mathrm{cells} \;\;\leq\;\; (\#S)^{2} d_{1}(d_{1}-d_{2})+\#S
(2d_{1}d_{2}-1)+d_2^2+1 
\end{equation*}
and, if $f$ is infinity complete, 
\begin{equation}  \label{eq:BoundDiffDeg}
\# \mathrm{cells} \;\;\leq\;\; \#S (d_{1}^{2}+d_{1}d_{2}-1)+d_2^2+1
\;\;\leq\;\; (d_1+d_2)(d_1^3 + d_1^2 d_2 - 2d_1^2 + d_1d_2^2 - 2d_1d_2 + 2).
\end{equation}
\end{theorem}

\begin{proof}
Similar to the proof of Theorem~\ref{thm:num_comp_equal}, B\'ezout's theorem
yields a bound of $((\# S)d_1 + d_2 + 1)^2$ which is reduced by the same
three cases. The first is the same, removing $\# S$ points. The second
removes $2 ((\#S) d_1 + d_2)$ points. The third becomes $d_1 d_2 (\#S)^2$
since $f_2$ also vanishes at every point in~$V_\mathbb{C}(f)$. Subtracting
these off yields the first inequality. The second inequality arises from
subtracting an additional $\# S (\# S -1) d_1 (d_1-d_2)$ and utilizing the
bound on $\#S$ from~Theorem~\ref{thm:num_S}.
\end{proof}

In the following, we give an example where the bound in Theorem~\ref{thm:num_comp} is exact.
\begin{example}
Consider the example (which was studied in Example~\ref{ex:21caseS}):
\begin{equation*}
\begin{split}
f_{1} & = x_1^2 + 3 x_1 x_2 - 5 x_2^2 + 9 x_1 - 2 x_2 + 40 \\
f_2 & = 5 x_1 - 2 x_2 + 1.
\end{split}%
\end{equation*}
\pichskip{0pt} 
\parpic[r][t]{
\begin{tikzpicture}[scale=0.15]
\tikzstyle{every node}=[font=\tiny]
\draw  (-10,-6.46)     --  (10,10.3);
\draw  (-10,-11.11)     --  (10,5.668);
\draw  (-10,5.2)     --  (10,0.439);
\draw  (-9.95,-2.469)     --  (9.95,-7.2);
\draw (-4.2, -9.9) -- (3.7,9.93584);
\filldraw[red]  (-1.94,-4.37)       circle (6pt);
\filldraw[red]  (0.85,2.63)       circle (6pt);
\filldraw[blue] (-6.27,-3.34)      circle (6pt);
\filldraw[blue] (5.17,1.59)      circle (6pt);
\node (1) at (-9,-4) {$1$};
\node (2) at (-6,1) {$2$};
\node (3) at (0,6) {$3$};
\node (4) at (3,6) {$4$};
\node (5) at (-6,-5) {$5$};
\node (6) at (-2,-1) {$6$};
\node (7) at (2,1) {$7$};
\node (8) at (6,5) {$8$};
\node (9) at (-4,-7) {$9$};
\node (10) at (0,-7) {$10$};
\node (11) at (6,-2) {$11$};
\node (12) at (9,3) {$12$};
\end{tikzpicture}
}
\noindent Note that $d_1=2$ and $d_2 = 1$. 
It turns out that $f$ is infinity complete.
From Theorem~\ref{thm:num_comp}, we have
\[
   \# \text{cells} \;\;\leq\;\; (d_{1}+d_{2})(d_{1}^{3}+d_{1}^{2}d_{2}-2d_{1}^{2}+d_{1}d_{2}^{2}-2d_{1}d_{2}+2)
    \;=\; 12.
\]
In the figure on the right, the red dots are the real solutions of $f$ and
the blue dots are elements of $S$. Note that there are 12 cells.
Thus, for this example, the upper bound for $\#$cells  is \emph{exact}.
\end{example}

\medskip

In the following, we will give two examples where the bound in Theorem~\ref{thm:num_comp} is not exact over the real numbers, but \textit{is} exact over the complex numbers in terms of counting critical points.
\begin{example}
Consider the example:
\begin{equation*}
\begin{split}
f_1 & =- 19x_1^3 - 37x_1^2x_2 + 43x_1x_2^2 + 46x_2^3 +49x_1^2 - 28x_1x_2 -
13x_2^2 + 22x_1 - 48x_2 + 10 \\
f_2 & = - 15x_1 - 6x_2 - 36
\end{split}%
\end{equation*}
\noindent Note that $d_1=3$ and $d_2 = 1$. 
It turns out that $f$ is infinity complete.
From Theorem~\ref{thm:num_comp}, we have
\[
   \# \text{cells} \;\;\leq\;\; (d_{1}+d_{2})(d_{1}^{3}+d_{1}^{2}d_{2}-2d_{1}^{2}+d_{1}d_{2}^{2}-2d_{1}d_{2}+2)
    \;=\; 68.
\]
A calculation with \texttt{Bertini}~\cite{BHSW06} 
confirms that there is exactly $68$ critical points over the complex numbers. 
In particular, of the 68 points, 46 are real (and thus routing points)
and the other 22 arise in complex conjugate~pairs.
\end{example}

\begin{example}
Consider the example:
\begin{equation*}
\begin{split}
f_{1} & = 30x_1^2x_2 - 50x_1x_2^2 + 36x_1^2 + 49x_1x_2 - 2x_2^2 + 50x_1 -
10x_2 + 37 \\
f_{2} & = 23x_1^2 + 37x_1x_2 + 14x_2^2 - 28x_1- 33x_2 - 8
\end{split}%
\end{equation*}
\noindent Note that $d_1=3$ and $d_2 = 2$. 
It turns out that $f$ is infinity complete.
From Theorem~\ref{thm:num_comp}, we have
\[
   \# \text{cells} \;\;\leq\;\; (d_{1}+d_{2})(d_{1}^{3}+d_{1}^{2}d_{2}-2d_{1}^{2}+d_{1}d_{2}^{2}-2d_{1}d_{2}+2)
    \;=\; 145.
\]
A calculation with \texttt{Bertini}~\cite{BHSW06} confirms 
that there is exactly $145$ critical points over the
complex numbers.
In particular, of the 145 points,
99 are real (and thus routing points) 
and the other 46 arise in complex conjugate pairs.
\end{example}

\section{Conclusion}

\label{sec:Conclusion} Newton homotopies are an approach to compute real
solutions to real polynomial systems. The set of real solutions obtained by
tracking from a start point depends upon which cell the start point is
selected from, with the boundaries of the cells being of particular interest
for analyzing the behavior of Newton homotopies. By focusing on the
bivariate case, we found a geometric description of the boundary,  an algebraic characterization of the
Zariski closure of the boundary, and characterized the local structure of the
boundary at singularities by looking at the tangent cone. From a bound on
the number of singular Newton homotopy curves, we derived an upper bound on
the number of cells in terms of the degrees of the polynomials and
demonstrated several examples where the bounds are exact. 
Future work involves extending to an arbitrary number of variables.

\bibliographystyle{abbrv}
\bibliography{references}

\end{document}